\documentclass{birkjour}
\usepackage{amsmath,amsthm,amscd,psfrag}
\pdfoutput=1 
\usepackage{graphicx,psfrag,url,color}
\usepackage{contour}
 \newtheorem{thm}{Theorem}[section]
 \newtheorem{cor}[thm]{Corollary}
 \newtheorem{lem}[thm]{Lemma}
 \theoremstyle{definition}
 \newtheorem{dft}[thm]{Definition}
 \theoremstyle{remark}
 \newtheorem{rem}[thm]{Remark}
 
 \numberwithin{equation}{section}

\begin{document}
%
%
%
%
\newcommand\pfad{} 
\newcommand\Thmref[1]{Theorem~\ref{#1}}
\newcommand\Lemref[1]{Lemma~\ref{#1}}
\newcommand\Corref[1]{Corollary~\ref{#1}}
\newcommand\Figref[1]{Figure~\ref{#1}}
\def\ol#1{\overline{#1}}
\def\wh#1{\widehat{#1}}
\def\RR{{\mathbb R}}
\def\ZZ{{\mathbb Z}}
\def\NN{{\mathbb N}}
\newcommand\Vkt[1]{{\mathbf #1}}
\def\wkl{\,<\mskip-10mu)\mskip4mu}
\def\Frac#1#2{{\displaystyle\frac{#1}{#2}}}
\def\smFrac#1#2{\mbox{\small $\displaystyle\frac{#1}{#2}$ }}
\def\ssmFrac#1#2{\mbox{\footnotesize $\displaystyle \frac{#1}{#2}$}}
\def\Sum{\displaystyle\sum}
\newcommand\name[1]{{\sc #1}}
\newcommand\Ecal{\mathcal E}
\newcommand\Hcal{\mathcal H}
\def\bow#1{\overset{\mbox{\raisebox{-0.8mm}{\Large $\frown$}}}{#1}}             
\def\zwi{\mskip 9mu}
\newlength{\nah}
\definecolor{blau}{cmyk}{1.00,0.30,0.00,0.00} 
\def\blue{\color{blau}}
\definecolor{rotcmyk}{cmyk}{0.00,1.00,1.00,0.00}
\def\red{\color{rotcmyk}}
\definecolor{gruen}{cmyk}{1.00,0.00,0.90,0.00} 
\def\green{\color{gruen}}
\sloppy



\title[The Geometry of Billiards in Ellipses and their Poncelet Grids]
 {The Geometry of Billiards in Ellipses\\ and their Poncelet Grids}

\author[H.\ Stachel]{Hellmuth Stachel} 
\address{Vienna University of Technology\\
Wiedner Hauptstr.\ 8-10/104\\
1040 Wien, Austria}

\email{stachel@dmg.tuwien.ac.at}


\date{April 2, 2021}

\begin{abstract}
The goal of this paper is an analysis of the geometry of billiards in ellipses, based on properties of confocal central conics.
The extended sides of the billiards meet at points which are located on confocal ellipses and hyperbolas.
They define the associated Poncelet grid.
If a billiard is periodic then it closes for any choice of the initial vertex on the ellipse. 
This gives rise to a continuous variation of billiards which is called billiard's motion though it is neither a Euclidean nor a projective motion.    
The extension of this motion to the associated Poncelet grid leads to new insights and invariants.
\end{abstract}

\maketitle

\section{Introduction}
A {\em billiard} is the trajectory of a mass point within a domain with ideal physical reflections in the boundary.
Already for two centuries, billiards in ellipses have attracted the attention of mathematicians, beginning with J.-V.\ Poncelet and A.\ Cayley.
One basis for the investigations was the theory of confocal conics.
In 2005 S.\ Tabachnikov published a book on billiards, based on the theory of completely integrable systems \cite{Tabach}.
In several publications and in the book \cite{DR_Buch}, V.\ Dragovi\'c and M.\ Radnovi\'c studied billiards, also in higher dimensions, from the viewpoint of dynamical systems.

Computer animations of billiards in ellipses, which were carried out by D.\ Reznik \cite{80}, stimulated a new vivid interest on this well studied topic, where algebraic and analytic methods are meeting (see, e.g., \cite{Ako-Tab,Bialy-Tab,Chavez,Reznik_surprise,Reznik_generic} and many further references in \cite{80}).
These papers focus on invariants of periodic billiards when one vertex varies on the ellipse while the caustic remains fixed.
This variation is called billiard motion though neither angles nor side lengths remain fixed and this is is also not a projective motion preserving the circumscribed ellipse.

The goal of this paper is a geometric analysis of billiards in ellipses and their associated Poncelet grid, starting from properties of confocal conics.
We concentrate on a certain symmetry between the vertices of any billiard and the contact points with the caustic, which can be an ellipse or hyperbola. 
Billiard motions induce motions of associated billiards with the same caustic and circumscribed confocal ellipses. 

\section{Metric properties of confocal conics} 

A family of {\em confocal} central conics (\Figref{fig:Beruehrpunkte}) is given by
\begin{equation}\label{eq:confocal}
  \frac{x^2}{a^2+k} + \frac{y^2}{b^2+k} = 1, \ \mbox{where} \
    k \in \RR \setminus \{-a^2, -b^2\}
\end{equation}
serves as a parameter in the family.
All these conics share the focal points 
\begin{equation}\label{eq:foci}
   F_{1,2} = (\pm d,0), \ \mbox{where} \ d^2:= a^2-b^2.
\end{equation}

The confocal family sends through each point $P$ outside the common axes of symmetry two orthogonally intersecting conics, one ellipse and one hyperbola \cite[p.~38]{Conics}.
The parameters $(k_e, k_h)$ of these two conics define the {\em elliptic coordinates} of $P$ with
\[  -a^2 < k_h < -b^2 < k_e\,.
\]
If $(x,y)$ are the cartesian coordinates of $P$, then $(k_e,k_h)$ are the roots of the quadratic equation
\begin{equation}\label{eq:cart_in_ell}
  k^2 + (a^2 + b^2 - x^2 - y^2)k + (a^2 b^2 - b^2 x^2 - a^2 y^2) = 0,
\end{equation}
while conversely
\begin{equation}\label{eq:ell_in_cart}
   x^2 = \frac{(a^2 + k_e)(a^2 + k_h)}{d^2}\,, \quad
    y^2 = -\frac{(b^2 + k_e)(b^2 + k_h)}{d^2}\,.
\end{equation}

Let $(a,b) = (a_c,b_c)$ be the semiaxes of the ellipse $c$ with $k = 0$.
Then, for points $P$ on a confocal ellipse $e$ with semiaxes $(a_e,b_e)$ and $k = k_e > 0$, i.e., exterior to $c$, the standard parametrization yields 
\begin{equation}\label{eq:P_coord}
 \begin{array}{c}
   P = (x,y) = (a_e\cos t,\,b_e\sin t), \ 0 \le t < 2\pi,
   \\[1.0mm]
   \mbox{with} \ a_e^2 = a_c^2 + k_e, \ b_e^2 = b_c^2 + k_e\,. 
 \end{array}
\end{equation}
For the elliptic coordinates $(k_e,k_h)$ of $P$ follows from \eqref{eq:cart_in_ell} that 
\[ k_e + k_h = a_e^2\cos^2 t + b_e^2\sin^2 t - a_c^2 - b_c^2.
\]
After introducing the respective tangent vectors of $e$ and $c$, namely
\def\arraycolsep{0.6mm}
\begin{equation}\label{eq:te_und_tc}
 \begin{array}{rcl}
  \Vkt t_e(t) &:= &(-a_e\sin t,\, b_e\cos t), 
  \\[0.8mm]   
  \Vkt t_c(t) &:= &(-a_c\sin t,\, b_c\cos t),
 \end{array}  \ \mbox{where} \
  \Vert \Vkt t_e\Vert^2 = \Vert \Vkt t_c\Vert^2 + k_e\,,   
\end{equation}
we obtain\footnote{
The norm $\Vert \Vkt t_e\Vert$ equals half length of the diameter of $e$ which is parallel to $\Vkt t_e\,$.}
\begin{equation}\label{eq:k_h}
   k_h = k_h(t) = -(a_c^2\sin^2 t + b_c^2\cos^2 t) = -\Vert\Vkt t_c(t)\Vert^2   
   = -\Vert\Vkt t_e(t)\Vert^2 + k_e
\end{equation}
and
\begin{equation}\label{eq:k_e minus k_h}
  \Vert\Vkt t_e(t)\Vert^2 = k_e - k_h(t)\,.
\end{equation}  
Note that points on the confocal ellipses $e$ and $c$ with the same parameter $t$ have the same coordinate $k_h$.
Consequently, they belong to the same confocal hyperbola (\Figref{fig:Poncelet_grid}).
Conversely, points of $e$ or $c$ on this hyperbola have a parameter out of $\{t, -t, \pi+t, \pi-t\}$.

\begin{figure}[htb] 
  \centering 
  \includegraphics[width=75mm]{\pfad 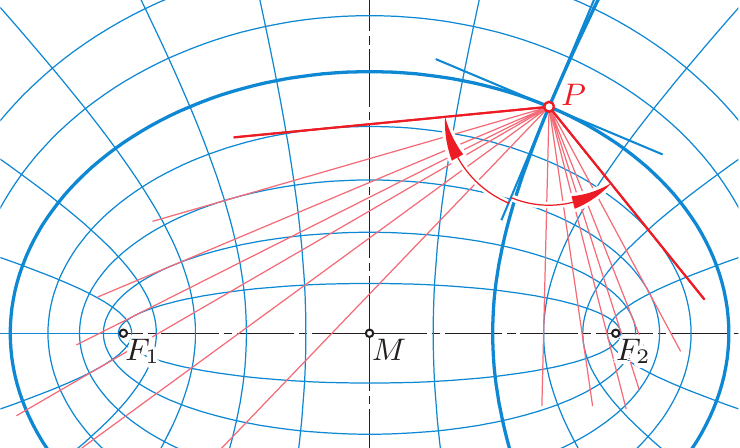} 
  \caption{The tangents from the point $P$ to the conics of a confocal family are symmetric w.r.t.\ the tangents at $P$ to the confocal conics passing through $P$.}
  \label{fig:Beruehrpunkte}
\end{figure}

Normal vectors of $e$ and $c$ can be defined respectively as 
\begin{equation}\label{eq:ne_und_nc}
 \begin{array}{rcl}
  \Vkt n_e(t) &:= &\Bigl(\Frac{\,\cos t}{a_e}, \,\Frac{\sin t}{b_e}\Bigr), 
  \\[2.8mm] 
  \Vkt n_c(t) &:= &\Bigl(\Frac{\,\cos t}{a_c}, \, \Frac{\sin t}{b_c}\Bigr),
 \end{array}  
   \quad \mbox{where} \ \Vert \Vkt n_c(t)\Vert 
    = \Frac{\Vert \Vkt t_c(t)\Vert}{a_c b_c}\,.
\end{equation}

We complete with two useful relations between the parameter $t$ and the second elliptic coordinate $k_h(t)$:
\begin{equation}\label{eq:st_ct}
  \tan^2 t = -\frac{b_c^2 + k_h(t)}{a_c^2 + k_h(t)} \quad\mbox{and}\quad
  \sin t \cos t = \frac{a_h b_h}{d^2}
\end{equation}   
with $a_h$ and $b_h$ as semiaxes of the hyperbola corresponding to the parameter $t$. i.e., $a_h^2 = a_c^2 + k_h$ and $b_h^2 = -(b_c^2 + k_h)$.

\medskip\noindent{\em Proof.} 
From \eqref{eq:k_h} follows
\[  k_h = -\frac{a_c^2\tan^2 t + b_c^2}{1 + \tan^2 t}, \quad\mbox{hence}\quad
  \tan^2 t(a_c^2 + k_h) = -b_c^2 - k_h
\]
and
\[  \sin t\cos t = \frac{\tan t}{1 + \tan^2 t} 
    = \frac{\sqrt{-(b_c^2 + k_h)(a_c^2 + k_h)}}{a_c^2 - b_c^2}
    = \frac{a_h\, b_h}{d^2}\,.                           \eqno{\qed} 
\]   

\medskip
Referring to \Figref{fig:Beruehrpunkte}, the following lemma addresses an important property of confocal conics 
(note, e.g., \cite[p.~38, 309]{Conics}).

\begin{lem}\label{lem:symm_involution} 
The tangents drawn from any fixed point $P$ to the conics of a confocal family share the axes of symmetry, which are tangent to the two conics passing through $P$.
\end{lem}

This means, if a ray is reflected at $P$ in one of the conics passing through, then the incoming and the outgoing ray contact the same confocal ellipse or hyperbola.

\begin{figure}[htb] 
  \centering 
  \includegraphics[width=75mm]{\pfad 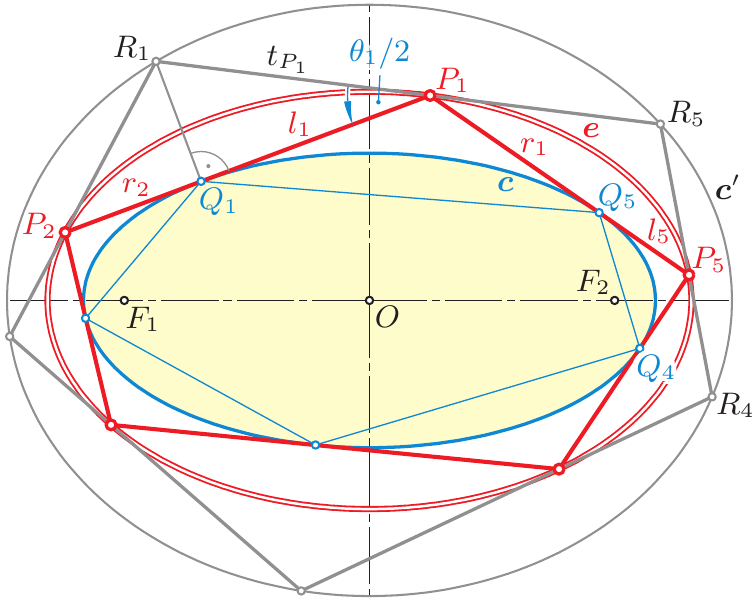}
  \caption{Periodic billiard $P_1 P_2 \dots P_5$ inscribed in the ellipse $e$ with the caustic $c$.}
  \label{fig:Affinitaet}
\end{figure}

Below, we report about results concerning a pair of confocal conics.
Due to their meaning for billiards in ellipses, we restrict ourselves to pairs $(e,c)$ of confocal ellipses with $c$ in the interior $e$, and we call $c$ the {\em caustic} (\Figref{fig:Affinitaet}).
 
\begin{lem}\label{lem:sin_theta} 
Let $P = (a_e\cos t,\,b_e\sin t)$ with elliptic coordinates $(k_e, k_h)$ be a point on the ellipse $e$ with $k_e > 0$ and $c$ be the confocal ellipse with $k=0$.
Then, the angle $\theta(t)/2$ between the tangent at $P$ to $e$ and any tangent from $P$ to $c$ satisfies
\begin{equation}\label{eq:Winkel/2} 
  \sin^2 \frac{\theta}2 
  = \frac{k_e}{\Vert\Vkt t_e(t)\Vert^2} = \frac{k_e}{k_e\!- k_h}\,, \
    \tan\frac{\theta}2 
  = \pm\sqrt{-\frac{k_e}{k_h}}\,, 
\end{equation} 
\begin{equation}\label{eq:Winkel}
  \cos\theta = 1 - \frac{2k_e}{\Vert \Vkt t_e(t)\Vert^2} 
  = \frac{k_h\! + k_e}{k_h\! - k_e}\,, \quad
  \sin\theta 
  = \pm \frac{2\sqrt{-k_e k_h}}{k_e - k_h}\,.  
\end{equation}
\end{lem}

\medskip\noindent{\em Proof.} 
The tangent $t_P$ to $e$ at $P = (a_e\cos t,\ b_e\sin t)$ in direction of $\Vkt t_e$ has the slope 
\[  f:= \tan\alpha_1 
    = \frac{-b_e\cos t}{a_e \sin t}\,.
\]
If $s_1$ and $s_2$ denote the slopes of the tangents from $P$ to $c$, then they satisfy
\[  y - b_e\sin t = s_i(x - a_e\cos t), \quad i=1,2.
\]
As tangents of $c$, their homogeneous line coordinates
\[  (u_0:u_1:u_2) = \left( (b_e\sin t - s_i\, a_e\cos t) : s_i : -1 \right)
\]
must satisfy the tangential equation $-u_0^2 + a_c^2 u_1^2 + b_c^2 u_2^2 = 0$ of $c$.
This results in a quadratic equation for the unknown $s$, namely 
\[  (a_e^2\sin^2 t - k_e) s^2 + 2a_eb_e s\sin t\cos t + (b_e^2\cos^2 t - k_e) = 0.
\]
We conclude
\[ s_1 + s_2 = \frac{-2a_eb_e\sin t\cos t}{a_e^2\sin^2 t - k_e} \zwi\mbox{and}\zwi   s_1 s_2 = \frac{b_e^2\cos^2 t - k_e}{a_e^2\sin^2 t - k_e}.
\]

The slopes $f = \tan\alpha_1$ of $t_P$ and $\tan\alpha_2 = s_1$ or $s_2$ of the tangents to $c$ imply for the enclosed signed angle $\theta(t)/2$ (for brevity, we often suppress the parameter $t$)
\[  \tan\frac{\theta}2 = \tan(\alpha_1 - \alpha_2) =
    \frac{s_1 - f}{1 + s_1f} = \frac{f - s_2}{1 + s_2f}\,,
\]
hence
\[ \tan^2 \frac{\theta}2 = \frac{(s_1 - f)(f - s_2)}{(1 + s_1f)(1 + s_2f)} 
    = \frac{f(s_1+s_2) - s_1 s_2 - f^2}{f(s_1 + s_2) + 1 + f^2 s_1 s_2}\,.
\]
After some computations, we obtain 
\[ \tan^2 \frac{\theta}2 = \frac{k_e}{a_e^2\sin^2 t + b_e^2\cos^2 t - k_e}
   = \frac{k_e}{\Vert\Vkt t_e\Vert^2 - k_e} = \frac{k_e}{\Vert\Vkt t_c\Vert^2}\,,
\]
therefore
\[  \cot^2 \frac{\theta}2 = \frac{\Vert\Vkt t_e\Vert^2}{k_e} - 1\zwi\mbox{and}\zwi
    \sin^2\frac{\theta}2 = \frac 1{1 + \cot^2 \frac{\theta}2} 
    = \frac{k_e}{\Vert\Vkt t_e\Vert^2}\,,
\]    
where $k_e = a_e^2 - a_c^2 = b_e^2 - b_c^2$, and finally 
\[ \cos\theta = 1 - 2\sin^2\frac{\theta}2 
   = 1 - \frac{2k_e}{\Vert\Vkt t_c\Vert^2 + k_e} 
   = \frac{\Vert\Vkt t_c\Vert^2 - k_e}{\Vert\Vkt t_c\Vert^2 + k_e}\,.
   \eqno{\qed}
\]

\medskip
\begin{rem}
A change of the origin $k=0$ for the elliptic coordinates in a family of confocal conics corresponds to a shift of the coordinates.
Hence, if in \Lemref{lem:sin_theta} the ellipse $c$ is replaced by another confocal  conic with the coordinate $k$, then the formulas \eqref{eq:Winkel/2} and \eqref{eq:Winkel} remain valid under the condition that we replace $k_e$ by $k_e - k$ and $k_h$ by $k_h - k$.
\end{rem}
 

\begin{lem}\label{lem:Hesse-NF} 
Let $P_1P_2$ be a chord of the ellipse $e$, which contacts the caustic $c$ at the point $Q_1$.
Then the signed distances of the line $[P_1,P_2]$ to the center $O$ and to the pole $R_1$ w.r.t.\ $e$ have the constant product $-k_e$.
The lines $[P_1,P_2]$ and $[Q_1,R_1]$ are orthogonal (\Figref{fig:Affinitaet}).
\end{lem}

\begin{proof}
Let the side $P_1P_2$ touch the caustic $c$ at the point $Q_1 = (a_c\cos t_1',\,b_c\sin t_1')$.
Then the Hessian normal form of the spanned line $t_Q = [P_1,P_2]$ reads
\[  t_Q\!: \ \frac{b_c \cos t_1'\,x + a_c\sin t_1'\,y - a_cb_c}{\sqrt{b_c^2\cos^2 t_1' 
    + a_c^2\sin^2 t_1'}} = 0.
\] 
Its pole w.r.t.\ $e$ has the coordinates
\begin{equation}\label{eq:R}
   R_1 = \left(\frac{a_e^2 \cos t_1'}{a_c}, \ 
   \frac{b_e^2 \sin t_1'}{b_c}\right).
\end{equation}
This yields for the signed distances to the line $t_Q$
\begin{equation}\label{eq:Ot}
  \ol{Ot_Q} = \frac{-a_c b_c}{\Vert\Vkt t_c(t_1')\Vert}
\end{equation}
and
\def\arraycolsep{1.0mm}
\begin{equation}\label{eq:Rt}
    \ol{R_1 t_Q} = \ol{R_1 Q_1}
    = \Frac{k_e(b_c^2\cos^2 t_1' + a_c^2\sin^2 t_1')}{a_c b_c\Vert\Vkt t_c(t_1')\Vert} = \Frac{k_e \Vert\Vkt t_c(t_1')\Vert}{a_c b_c}\,.
\end{equation}
Thus, we obtain a constant product $\ol{O t_Q}\cdot \ol{R_1 t_Q} = -k_e\,$, as stated
above.   
\\
The last statement holds since $R_1$ and $Q_1$ are the poles of $[P_1,P_2]$ w.r.t.\ the confocal conics $e$ and $c$.
It is wellknown that the poles of any line $\ell$ w.r.t.\ a family of confocal conics lie on a line $\ell^*$ orthogonal to $\ell$ (see, e.g., \cite[p.~340]{Conics}).
\end{proof} 

\section{Confocal conics and billiards} 

By virtue of \Lemref{lem:symm_involution}, all sides of a billiard inscribed to the ellipse $e$ with parameter $k=k_e$ are tangent to a fixed conic $c$ confocal with $e$ (\Figref{fig:Affinitaet}).
This conic $c$ with parameter $k_c$ is called {\em caustic}.
It can be a smaller ellipse with $-b^2 < k_c < k_e$ or a hyperbola with $-a^2 < k_c < -b^2$ or, in the limiting case with $k_c = -b^2$, consist of the pencils of lines with the focal points $F_1, F_2$ of $c$ as carriers.
At the beginning, we confine ourselves to an ellipse as caustic with $k_c = 0$ (\Figref{fig:Affinitaet}), and we speak of {\em elliptic} billiards.
Only the Figures~\ref{fig:hyp_Kaust}, \ref{fig:hyp_Kaust3} and \ref{fig:hyp_Kaust2} show examples of (periodic) billiards in $e$ with a hyperbola as caustic, called {\em hyperbolic} billards. 

For billiards $\dots P_1P_2P_3\dots $ in the ellipse $e$ and with the ellipse $c$ as caustic, we assume from now on a counter-clockwise order and signed exterior angles $\theta_1,\theta_2,\theta_3,\dots$ (see \Figref{fig:Affinitaet}).
The tangency points $Q_1,Q_2,\dots$ of the billiard's sides $P_1P_2, P_2P_3,\dots$ with the caustic $c$ subdivide the sides into two segments.
We denote the lengths of the segments adjacent to $P_i$ as
\begin{equation}\label{eq:def_rili}
  l_i := \ol{P_iQ_i} \zwi\mbox{and}\zwi r_i:= \ol{P_iQ_{i-1}}. 
\end{equation}
Based on the parametrizations $(a_e\cos t,\,b_e\sin t)$ of $e$ and $(a_c\cos t',\, b_c\sin t')$ of $c$, we denote the respective parameters of $P_1$, $Q_1$, $P_2$, $Q_2$, $P_3,\dots$ with $t_1$, $t_1'$, $t_2$, $t_2'$, $t_3, \dots$ in strictly increasing order.

The following two lemmas deal with sides of billiards in the ellipse $e$.

\begin{lem}\label{lem:2-2-correspondence} 
The connecting line $[P_i, P_{i+1}]$ of the vertices with respective parameters $t_1, t_2$ on $e$ contacts the caustic $c$ if and only if
\[  \frac{a_c^2}{a_e^2}\,\cos^2\frac{t_1+t_2}2 
    + \frac{b_c^2}{b_e^2}\,\sin^2\frac{t_1+t_2}2 = \cos^2\frac{t_1-t_2}2. 
\]
This is equivalent to
\[  \sin^2\frac{t_1-t_2}2 = \frac{k_e}{a_eb_e}\,\Bigl\Vert
    \,\Vkt t_e\Bigl(\smFrac{t_1+t_2}2\Bigr)\Bigr\Vert^2.
\]
\end{lem}

\begin{proof}
The line connecting the points $(a_e\cos t_i,\,b_e\sin t_i)$, $i=1,2$, has  homogeneous line coordinates $(u_0:u_1:u_2)$ equal to
\[ 
   \left( a_e b_e (\cos t_1 \sin t_2 - \sin t_1 \cos t_2) : 
     b_e(\sin t_1 - \sin t_2) : a_e(\cos t_2 - \cos t_1)\right). 
\]
It contacts the caustic $c$ if $-u_0^2 + a_c^2 u_1^2 + b_c^2 u_2^2 = 0$, 
i.e.,
\[ \begin{array}{l}
    a_c^2b_e^2 \sin^2\Frac{t_1 - t_2}2 \cos^2\Frac{t_1 + t_2}2 
     + b_c^2a_e^2 \sin^2\Frac{t_1 - t_2}2 \sin^2\Frac{t_1 + t_2}2 
    \\[2.5mm]
    = a_e^2 b_e^2 \sin^2\Frac{t_2 - t_1}2 \cos^2\Frac{t_2 - t_1}2\,.
  \end{array}
\]
Under the condition $\sin[(t_1-t_2)/2]\ne 0$ we obtain the first claimed equation.
The second follows after the substitutions $a_c^2 = a_e^2 - k_e$ and $b_c^2 = b_e^2 - k_e$ from
\[  1 - \frac{k_e}{a_e^2 b_e^2}\left( b_e^2\cos^2\frac{t_1 + t_2}2 
    + a_e^2\sin^2\frac{t_1 + t_2}2\right) = \cos^2\frac{t_2 - t_1}2\,. 
\]
by \eqref{eq:ne_und_nc}.
\end{proof} 

\begin{lem}\label{lem:between} 
Referring to the notation in \Lemref{lem:2-2-correspondence}, if the side $P_iP_{i+1}$ contacts the caustic $c$ at $Q_i$ with parameter $t_i'$, then 
\[  \sin t_i' = \frac{b_c}{b_e}\,\frac{\sin\frac{t_i+t_{i+1}}2}{\cos\frac{t_i-t_{i+1}}2}\,, \
    \cos t_i' = \frac{a_c}{a_e}\,\frac{\cos\frac{t_i+t_{i+1}}2}{\cos\frac{t_i-t_{i+1}}2}\,,
    \ \tan t_i' = \frac{b_c a_e}{a_c b_e}\,\tan\frac{t_i+t_{i+1}}2\,.
\]
\end{lem}   

\begin{proof}
The tangent to $c$ at $Q_1$ has the line coordinates
\[  (u_0:u_1:u_2) = \left( -a_cb_c : b_c\cos t_1': a_c\sin t_1'\right),
\]  
which must be proportional to those in the proof of \Lemref{lem:2-2-correspondence}.
\end{proof}

\begin{rem}
The half-angle substitution
\[  \tau_i:= \tan\frac{t_i}2 \quad \mbox{for} \ i=1,2
\]
allows to express the equation in \Lemref{lem:2-2-correspondence} (for $i=1$) in projective coordinates on $e$.
We obtain a symmetric biquadratic condition 
\[ b_e^2 k_e \tau_1^2\tau_2^2 - b_c^2 a_e^2(\tau_1^2 + \tau_2^2)
   + 2(a_e^2 k_e + a_c^2 b_e^2)\tau_1\tau_2 + b_e^2 k_e = 0,
\]
which defines a 2-2-correspondence on $e$ between the endpoints $P_1, P_2$ of a chord which contacts $c$.
This remains valid after iteration, i.e., between the initial point $P_1$ and the endpoint $P_{N+1}$ of a billiard after $N$ reflections in $e$. 
\\
Now, we recall a classical argument for the underlying Poncelet porism (see also \cite{Halbeisen} and the references there):
A 2-2-correspondence different from the identity keeps fixed at most four points.
However, four fixed points on $e$ are already known as contact points between $e$ and the common complex conjugate (isotropic) tangents\footnote{
They follow from the second equation in \Lemref{lem:2-2-correspondence} from $t_1 = t_2$.}
 with the caustic $c$, since tangents of $e$ remain fixed under the reflection in $e$.
If therefore one $N$-sided billiard in $e$ with caustic $c$ closes, then the correspondence is the identity and all these billiards close. 
\end{rem}

\begin{figure}[t] 
  \centering 
  \includegraphics[width=80mm]{\pfad 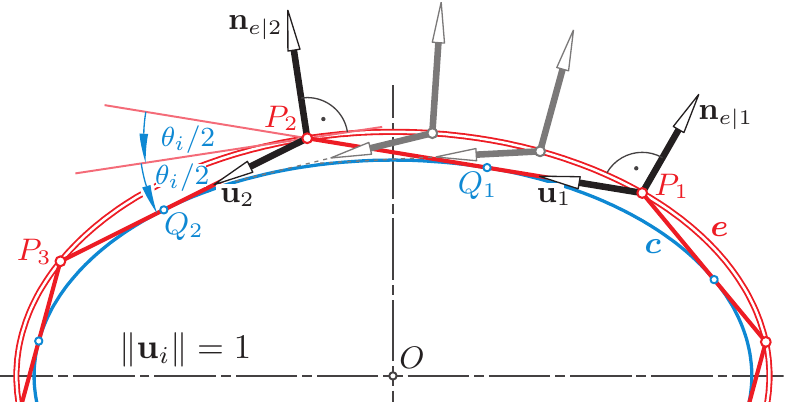} 
  \caption{The Joachimsthal integral $J_e:= -\langle \Vkt u_i,\,\Vkt n_{e|i} \rangle$ is constant along $e$.}  
  \label{fig:Joachimsthal}
\end{figure}

\medskip
Given any billiard $P_1P_2\dots$ in the ellipse $e$, let $\Vkt p_i = (x_i,\,y_i)$ denote the position vector of $P_i$ for $i=1,2,\dots$, while $\Vkt u_1, \Vkt u_2,\dots$ denote the unit vectors of the oriented sides $P_1P_2$, $P_2P_3$, \dots\,.
By \eqref{eq:ne_und_nc}, the vector $\Vkt n_{e|i}:= (x_i/a_e^2,\ y_i/b_e^2)$ is orthogonal to $e$ at $P_i$.
According to \cite[Proposition~2.1]{Ako-Tab}, the scalar product 
\begin{equation}\label{eq:J}
  J_e:= -\langle \Vkt u_i,\,\Vkt n_{e|i}\rangle 
\end{equation}
is invariant along the billiard in $e$ and called {\em Joachimsthal integral} (note also \cite[p.~54]{Tabach}).

The invariance of the Joachimsthal integral, which also holds in higher dimensions for billiards in quadrics, is the key result for the integrability of billiards, i.e., in the planar case for the existence of a caustic \cite[p.~3]{Ako-Tab}.
In our approach, the invariance of $J_e$ follows from \Lemref{lem:sin_theta}.

\begin{lem}\label{lem:Tab_invariant}
The Joachimsthal integral $J_e:= -\langle\Vkt u_i,\,\Vkt n_{e|i}\rangle$ equals 
\[ J_e = \frac{\sqrt{k_e}}{a_eb_e}
\]
with $k_e$ as elliptic coordinate of $e$ w.r.t.\ $c$, i.e.,
$k_e = a_e^2 - a_c^2 = b_e^2 - b_c^2$.
\end{lem}

\begin{proof} 
From \eqref{eq:Winkel/2} follows for the points $(a_e\cos t, \ b_e\sin t)$ of $e$ 
\[  J_e = -\langle\Vkt u,\,\Vkt n_e\rangle 
    = -\cos \Bigl(\frac{\pi}2 + \frac{\theta}2\Bigr)\,
     \Vert \Vkt n_e\Vert = \sin\frac{\theta}2\,\Vert\Vkt n_e\Vert
    = \sin\frac{\theta}2\,\frac{\Vert\Vkt t_e\Vert}{a_eb_e}\,,
\]
hence by \eqref{eq:Winkel/2}, \eqref{eq:k_h}, \eqref{eq:te_und_tc}, and \eqref{eq:ne_und_nc} 
\[ J_e^2 =  \sin^2 \frac{\theta}2\,\Vert\Vkt n_e\Vert^2 
   = \frac{k_e}{\Vert\Vkt t_e\Vert^2}\,\Vert\Vkt n_e\Vert^2 
   = \frac{k_e}{a_e^2\,b_e^2}\,.
\]
This confirms the claim.  
\end{proof}

\subsection{Poncelet grid}

The following theorem is the basis for the {\em Poncelet grid} associated to each billiard.
We formulate and prove a projective version. 
The special case dealing with confocal conics, has already been published by Chasles \cite[p.~841]{Chasles} and later by B\"ohm in \cite[p.~221]{Boehm1}.
The same theorem was studied in \cite{Schwartz} and in \cite{Akopyan}.
In \cite{Izmestiev}, the authors proved it in a differential-geometric way.

\medskip
In the theorem and proof below, the term {\em conic} stands for regular dual conics, i.e., conics seen as the set of tangent lines, but also for pairs of line pencils and for single line pencils with multiplicity two.
Expressed in terms of homogeneous line coordinates, the corresponding quadratic forms have rank 3, 2 or 1, respectively.
Moreover, we use the term {\em range} for a pencil of dual conics.
The term {\em net} denotes a 2-parametric linear system of dual curves of degree 2. 
Obviously, conics and ranges included in a net play the role of points and lines of a projective plane within the 5-dimensional projective space of dual conics.
Any two ranges in a net share a conic (compare with \cite[Th\'eor\`emes I\,--\,IV]{Chasles2}).

\begin{figure}[t] 
  \centering 
  \includegraphics[width=115mm]{\pfad 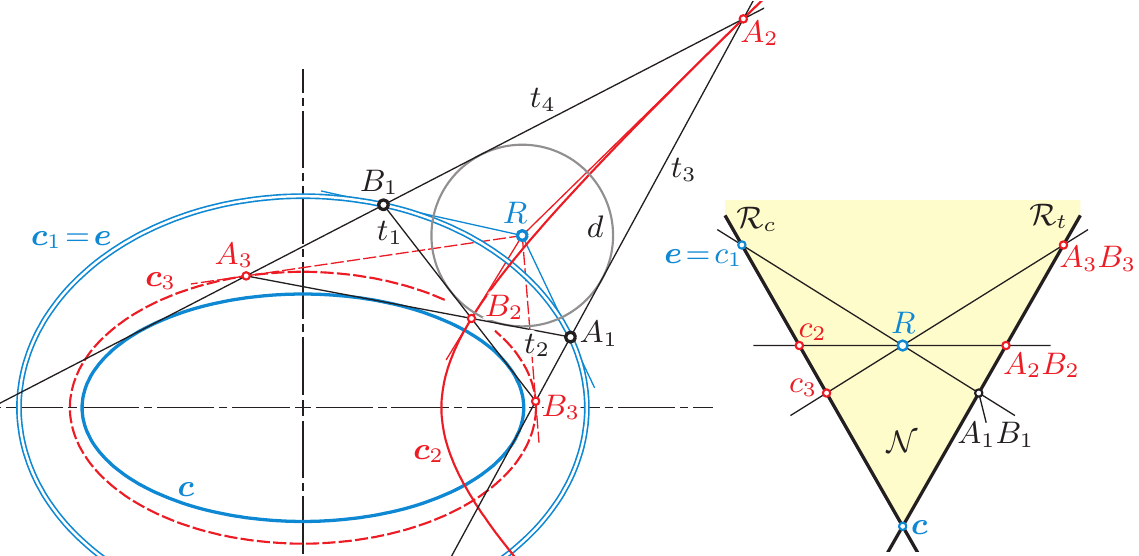} 
  \caption{Left: Opposite vertices $A_i,B_i$ of the quadrilateral $t_1\dots t_4$ of tangents to $c$ belong to a conic $c_i$ out of the range ${\mathcal R}_c$ (Theorem~\ref{Thm_1}).
 Right: The net $\mathcal N$ and the ranges ${\mathcal R}_t$, ${\mathcal R}_c$ in the projective space of dual conics.}
  \label{fig:Chasles}
\end{figure}

\begin{thm}\label{Thm_1} 
Let $c$ be a regular conic and $A_1, B_1$ two points such that the tangents $t_1,\dots,t_4$ drawn from $A_1$ and $B_1$ to $c$ form a quadrilateral. 
Its remaining pairs of opposite vertices are denoted by $(A_i,B_i)$, $i=2,3$. 
Then, 

\begin{enumerate}
\item for each conic $c_1$ passing through $A_1$ and $B_1$, the range ${\mathcal R}_c$ spanned by $c$ and $c_1$ contains conics $c_i$ passing through $A_i$ and $B_i$, simultaneously.
The tangents at $A_j$ and $B_j$ to $c_j$ for $j=1,2,3$ meet at a common point $R$. 
If $c_i$ has rank 2, then we obtain, as the limit of $c_i\,$, the diagonal $[A_i,B_i]$ of the quadrilateral $t_1,\dots,t_4$.  

\item This result holds also in the limiting case $t_1 = t_2$, where the chord $A_1B_1$ of $c_1$ contacts $c$ at $B_2$.
\end{enumerate}
\end{thm}

In \Figref{fig:Chasles}, the particular case is displayed where $c$ and $c_1 = e$ span a range $\mathcal R_c$ of confocal conics (note also \cite[Fig.~19]{Bob_Fairly}).
Then by \Lemref{lem:symm_involution}, the tangents at $A_j$ and $B_j$ to $c_j$ are angle bisectors of the quadrilateral.
In case of a rank deficiency of $c_i\,$, either one axis of symmetry of the confocal family or the line at infinity shows up as $c_i\,$. 

\begin{proof}
The conics tangent to $t_1,\dots,t_4$ define a range $\mathcal R_t$, which includes for $j = 1,2,3$ the pairs of line pencils $(A_j,B_j)$ as well as the initial conic $c$.
On the other hand, $c$ and $c_1$ span a range $\mathcal R_c$.
Since both ranges share the conic $c$, they span a net $\mathcal N$ of conics.

The pair $(A_1,B_1)$ of line pencils spans together with $c_1$ the range of conics
sharing the points $A_1, B_1$ and the tangents there, which meet at a point $R$.
This range, which also belongs to $\mathcal N$, contains the rank-1 conic with carrier $R$.
Each pair of line pencils $(A_i,B_i)$, $i = 1,2\,$, spans with the pencil $R$ again a range within $\mathcal N$.
This range shares with the range $\mathcal R_c$ a conic $c_i$ passing through $A_i$ and $B_i$ with respective tangent lines through $R$.\footnote{
An extended version of this theorem in \cite{MonGeom} addresses the symmetry between the ranges $\mathcal R_c$ and $\mathcal R_t$. 
This generalizes the statement that in the case of confocal conics $c$ and $c_1$ the quadrilateral $A_1A_2B_1B_2$ has an incircle  $d$ (Figures~\ref{fig:Chasles},  \ref{fig:Inkreise} and \ref{fig:Inkreise2}, compare with \cite{Akopyan, Izmestiev}).}

All these conclusions remain valid in the case, when $\mathcal R_t$ consists of conics which touch $c$ at $B_2$ and are tangent to $t_3$ and $t_4\,$.
\end{proof}

\medskip
As already indicated by the notation, we are interested in the particular case of \Thmref{Thm_1} where the conics $c$ and $e$ in the range $\mathcal R_c$ are confocal.
The following result follows directly from \Thmref{Thm_1} and summarizes properties of the Poncelet grid.
For the points of intersection between extended sides of a billiard $\dots P_0 P_1 P_2 \dots$ we use the notation
\begin{equation}\label{eq:S_i^j}
  S_i^{(j)}:= \left\{ \begin{array}{rl}
  [P_{i-k-1},P_{i-k}]\cap[P_{i+k},P_{i+k+1}] \zwi &\mbox{for} \zwi j = 2k,
   \\[0.6mm] 
  [P_{i-k},P_{i-k+1}]\cap[P_{i+k},P_{i+k+1}] \zwi &\mbox{for} \zwi j = 2k-1
 \end{array} \right.  
\end{equation}
where $i = \dots,0,1,2,\dots$ and $j= 1,2,\dots$
Note that there are $j$ sides between those which intersect at $S_i^{(j)}$, and `in the middle' of these $j$ sides there is for even $j$ the vertex $P_i$ and otherwise the point of contact $Q_i\,$.
At the same token, the point $S_i^{(j)}$ is the pole of the diagonal $[Q_{i-k-1},Q_{i+k}]$ or $[Q_{i-k},Q_{i+k}]$ of the polygon $\dots Q_1Q_2Q_3\dots$ of contact points w.r.t.\ the caustic. 

\begin{thm}\label{thm:symmetry} 
Let $\dots P_0 P_1 P_2 \dots$ be a billiard in the ellipse $e$ with sides $P_iP_{i+1}$ which contact the ellipse $c$ at the respective points $Q_i$ for all $i\in\ZZ$.
Then the vertices $S_i^{(j)}$ of the associated Poncelet grid are distributed on the following conics.  

\begin{enumerate}
\item The points $S_i^{(1)}$, $S_i^{(3)}, \dots$ 
are located on the confocal hyperbola through $Q_i$, while the points
$S_i^{(2)}$, $S_i^{(4)},\dots$ 
are located on the confocal hyperbola through $P_i$.

\item For each $j\in\{1,2,\dots\}$, the points $\dots S_i^{(j)} S_{i+(j+1)}^{(j)}
S_{i+2(j+1)}^{(j)}\dots$ are vertices of another billiard with the caustic $c$ inscribed in a confocal ellipse $e^{(j)}$, provided that $e^{(j)}$ is regular.
Otherwise $e^{(j)}$ coincides with an axis of symmetry or with the line at infinity. 
The locus $e^{(j)}$ is independent of the position of the initial vertex $P_0 \in e$. 
\end{enumerate}
\end{thm}

\begin{figure}[htb] 
  \centering 
  \includegraphics[width=115mm]{\pfad 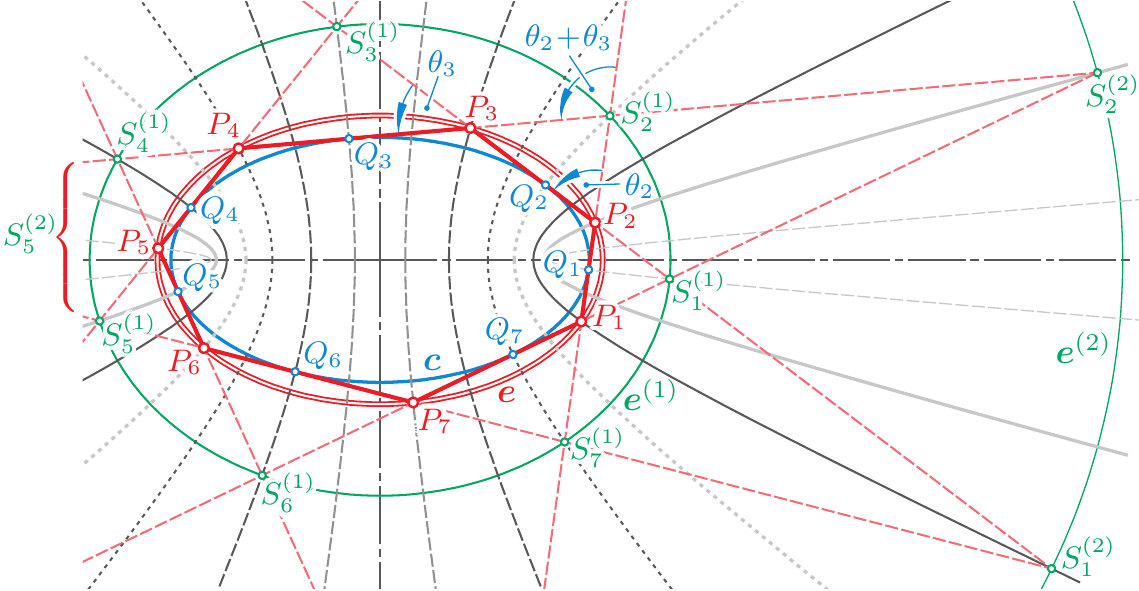} 
  \caption{Periodic billiard $(N\!= 7)$ with extended sides.}
  \label{fig:Poncelet_grid}
\end{figure}

\begin{proof}
1.\ The side lines $[P_0,P_1]$ ($P_0\!=\!P_7$ in \Figref{fig:Poncelet_grid}) and $[P_2,P_3]$ meet at $S_1^{(1)}$, while $[P_1,P_2]$ contacts $c$ at $Q_1$.
By \Thmref{Thm_1},\,2.\ the points $Q_1$ and $S_1^{(1)}$ belong to the same confocal hyperbola.
\\
Now we go one step away from $Q_1$: the tangents from $P_0$ and $P_3$ to $c$ intersect at $S_1^{(1)}$ and $S_1^{(3)} = [P_{-1},P_0]\cap[P_3,P_4]$.
The confocal conic through $S_1^{(1)}$ and $S_1^{(3)}$ must again be the hyperbola through $Q_1$.
This follows by continuity after choosing $Q_1$ on one axis of symmetry.  
Iteration confirms the first claim in \Thmref{thm:symmetry}.
\\
The tangents to $c$ from $P_0$ and $P_2$ form a quadrilateral with $P_1$ and $S_1^{(2)}$ as opposite vertices.
Therefore, there exists a confocal conic passing through both points.
This conic must be a hyperbola, as can be concluded by continuity: 
If $P_1$ is specified at a vertex of $e$, then due to symmetry the points $P_1$ and $S_1^{(2)}$ are located on an axis of symmetry.
\\
The tangents to $c$ from $P_{-1}$ and $P_3$ form a quadrilateral with $S_1^{(2)}$ and $S_1^{(4)}$ as opposite vertices.
\Thmref{Thm_1} and continuity guarantee that this is again the confocal hyperbola through $P_1$.
Iteration shows the same of $S_1^{(6)}$ etc.
However, the points $P_1$, $S_1^{(2)}$, $S_1^{(4)},\dots$ need not belong to the same branch of the hyperbola.

\smallskip
\noindent 2.\ The tangents through $P_2$ and $S_2^{(2)}$ (note \Figref{fig:Poncelet_grid}) form a quadrilateral with $S_1^{(1)}$ and $S_2^{(1)}$ as opposite vertices.
This time, continuity shows that the two points belong to the same confocal ellipse $e^{(1)}$.
The same holds for the tangents through $P_3$ and $S_3^{(2)}$ etc. 
\\
Similarily, starting with the points $P_0$ and $P_3$, we find the ellipse $e^{(2)}$ through $S_1^{(2)}$ and $S_2^{(2)}$, and so on.

\noindent
In order to prove that these ellipses $e^{(1)}, e^{(2)}, \dots$ are independent of the choice of the initial point $P_1\in e$, we follow an argument from \cite[proof of Corollary~2.2]{Ako-Tab}:
The claim holds for all confocal ellipses $e$ where billiards with the same caustic $c$ are aperiodic and traverse $e$ infinitely often.
Since these ellipses form a dense set, the claim holds also for those with periodic billiards.
The invariance of the ellipses $e^{(1)}, e^{(2)}, \dots$ is already mentioned in 
\cite[Theorem~7]{Ako-Tab}.

\noindent
An alternative proof consists in demonstrating that the elliptic coordinate $k_e^{(j)}$ of $e^{(j)}$ for all $j$ does not depend on the parameter $t$.
As one example, we present below in \eqref{eq:ae1}, \eqref{eq:be1} and \eqref{eq:ke1} formulas for the semiaxes $a_{e|1}$, $b_{e|1}$ and the elliptic coordinate $k_{e|1}$ of $e^{(1)}$.
\end{proof}

\begin{rem}\label{rem:zigzag} 
\Figref{fig:Poncelet_grid} reveals, that the polygons $P_1 S_1^{(1)} P_2 S_2^{(1)}\dots$ as well as $P_1 S_2^{(2)} P_3 S_4^{(2)}\dots$ and $S_1^{(2)} S_1^{(1)} S_2^{(2)} S_2^{(1)}\dots$ are zigzag billiards in rings bounded by two confocal ellipses.
However, we find also zigzag billiards between two confocal hyperbolas, e.g.,  
$\dots S_1^{(2)} P_2 P_1 S_2^{(2)}\dots$ or the twofold covered $\dots S_1^{(2)}$  $S_1^{(1)} P_1 Q_1 P_1 S_1^{(1)} S_1^{(2)}\dots\,$.
Billiards between other combinations of confocal conics can be found in \cite{DR_russ}.
\end{rem}

\begin{figure}[b] 
  \centering 
  \includegraphics[width=108mm]{\pfad 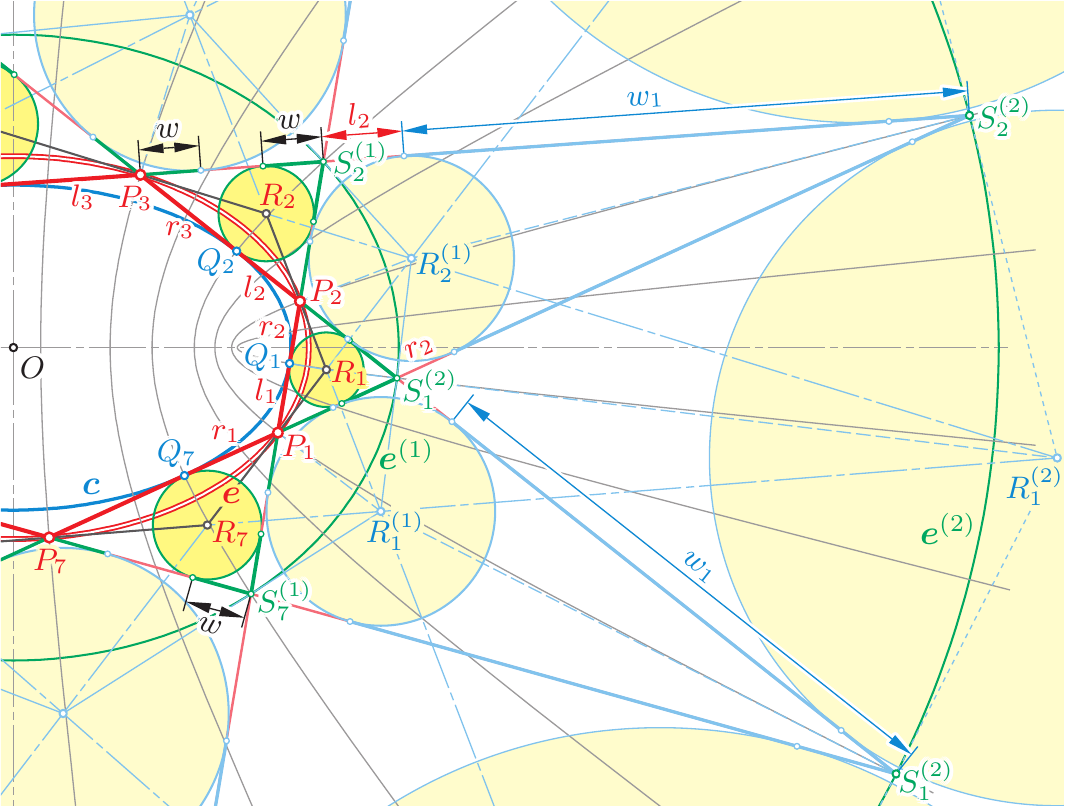} 
  \caption{The power $w^2$ of $S_2^{(1)}$ w.r.t.\ the incircle of the triangle $S_2^{(1)}P_2P_3$ shows up at all $S_i^{(1)}$ and equals the power of $P_i$ w.r.t.\ the incircle of the quadrangle $P_i S_{i-1}^{(1)} S_i^{(2)} S_i^{(1)}$.}
  \label{fig:Inkreise}
\end{figure}

The coming lemma addresses invariants related to the incircles of quadrilaterals built from the tangents to $c$ from any two vertices $P_i$ and $P_j$ of a billiard in $e$ (note circle $d$ in \Figref{fig:Chasles} and \cite{Akopyan, Bobenko, Izmestiev}).

\begin{lem}\label{lem:w=konst} 
Referring to \Figref{fig:Inkreise}, the power $w^2$ of the point $S_i^{(1)}$ w.r.t.\ the incircle of the triangle $P_i P_{i+1} S_i^{(1)}$  is the same for all $i$. 
Similarly, the power $w_1^2$ of $S_i^{(2)}$ w.r.t.\ the incircle of the quadrangle $P_i S_{i-1}^{(1)} S_i^{(2)} S_i^{(1)}$ is constant. 
\end{lem}

\begin{proof}
According to Graves's construction \cite[p.~47]{Conics}, an ellipse $e$ can be constructed from a smaller ellipse $c$ in the following way:
Let a closed piece of string strictly longer than the perimeter of $c$ be posed around $c$.
If point $P$ is used to pull the string taut, then $P$ traces a confocal ellipse $e$. 
Consequently, for each vertex $P_i$ and neighboring tangency points $Q_{i-1}$ and $Q_i$ of a billiard in $e$ with caustic $c$, the sum of the lengths $\ol{Q_{i-1}P_i}$ and $\ol{P_iQ_i}$ minus the length of the elliptic arc between $Q_{i-1}$ and $Q_i$, i.e.,
\begin{equation}\label{eq:De}
   D_e:= \ol{Q_{i-1}P_i} + \ol{P_iQ_i} - \bow{Q_{i-1}Q_i}
\end{equation}   
is constant (\Figref{fig:Inkreise}).

\noindent
The incircle of $P_2 P_3 S_2^{(1)}$ has the center $R_2$ and the radius $\ol{Q_2R_2}$ by \eqref{eq:Rt}.
The power of $P_2$ w.r.t.\ this circle is $l_2^2$, that of $P_3$ is $r_3^2$.
From Graves' construction follows for the ellipse $e$ that 
\[  D_e = r_2 + l_2 - \bow{Q_1Q_2} = r_3 + l_3 - \bow{Q_2Q_3} = \mbox{const.} 
\] 
is the same for all $P_i$, provided that $\bow{Q_iQ_j}$ denotes the length of the (shorter) arc along $c$ between $Q_i$ and $Q_j$.
For the analogue invariant at $e^{(1)}$ follows (\Figref{fig:Inkreise})
\begin{equation}\label{eq:De1,2}
 \begin{array}{rcl}
    D_{e|1} &:= &\ol{Q_1S_2^{(1)}} + \ol{S_2^{(1)}Q_3} - \bow{Q_1Q_3} 
    \\[0.5mm]
    &= &(r_2 + l_2 + w) + (r_3 + l_3 + w) - \bow{Q_1Q_3}
         = 2D_e + 2w = \mbox{const.}, 
  \end{array}
\end{equation} 
hence $w = \mbox{const.}$, where $w^2$ is the power of $S_2^{(1)}$ w.r.t.\ the said incircle. 

\noindent
Since the incircle of the quadrangle $P_2 S_1^{(1)} S_2^{(2)} S_2^{(1)}$ is an excircle of the triangle $P_2 P_3 S_2^{(1)}$ (\Figref{fig:Inkreise}), the power of $P_2$ w.r.t.\ the excircle equals $w^2$, too.
This follows from elementary geometry.

\noindent
As an alternative, the constancy of $w$ can also be concluded from the fact,
that neighboring circles with centers $R_i^{1},R_i$ or $R_i,R_{i+1}^{1}$ share three tangents, and one circle is an incircle, the other an excircle of the triangle.
Therefore, on the common side the same lengths $w$ shows up twice and also at the adjacent pairs of neighboring circles.
Since the distance $w$ is constant for all aperiodic billiards, it reveals also for periodic billiards that $w$ is independent of the choice of the initial vertex.
In a similar way follows the invariance of the lengths $w_1$, as shown in \Figref{fig:Inkreise}. 
\end{proof}

It needs to be noted that $S_2^{(1)}$ or $S_2^{(2)}$ can be located on the other branch of the related hyperbola.
Then the said `incircle' of the triangle $P_2 P_3 S_2^{(1)}$ has to be replaced by the excircle which contacts $c$ at $Q_2$ and is tangent to the side lines $[P_1,P_2]$ and $[P_3,P_4]$.
Similarly, the said 'incircle' of the quadrangle becomes an excircle.
In all these cases, \Lemref{lem:w=konst} and the proof given above have to be adapted. 

\begin{figure}[t] 
  \centering 
  \includegraphics[width=85mm]{\pfad 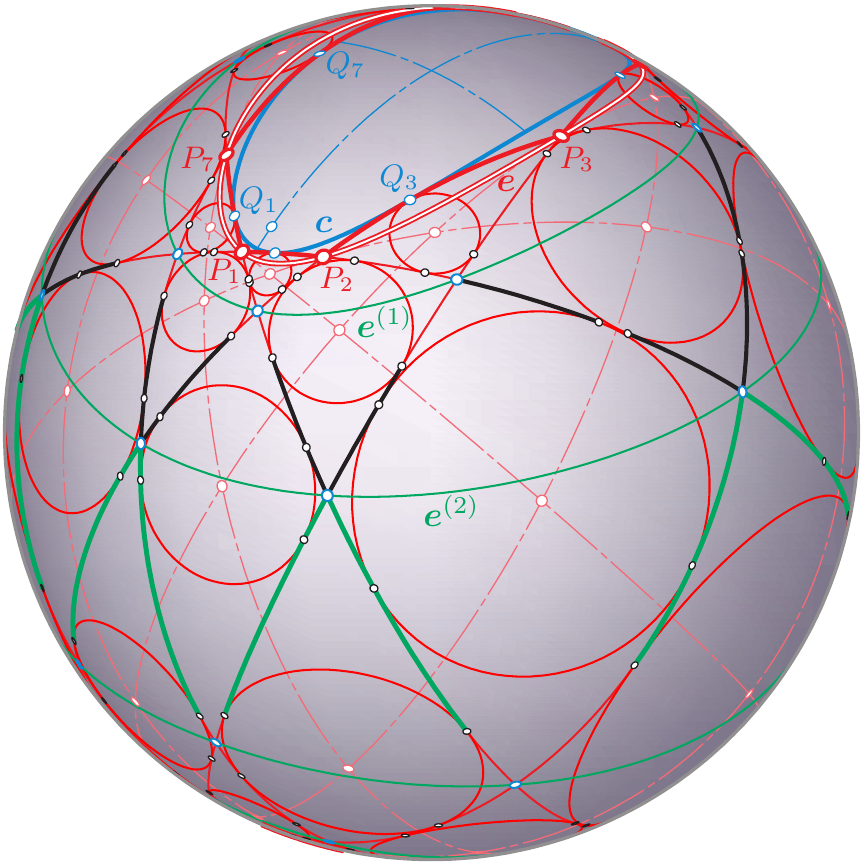}
  \caption{Periodic billiard $P_1P_2\dots P_7$ on the sphere with extended sides and their contact points with the incircles. 
All circular arcs marked in black have the same length, as well as those marked in green, and both lengths are invariant against changes of $P_1$ on $e$.}
  \label{fig:sph_inkreise}
\end{figure}

\begin{rem}
The Theorems \ref{Thm_1} and \ref{thm:symmetry} as well as the constancy of the length $w$ according to \Lemref{lem:w=konst} are also valid in spherical geometry (note \cite{MonGeom}) and in hyperbolic geometry.
On the sphere (see \Figref{fig:sph_inkreise}), the caustic consists of a pair of opposite components, and for $N$-periodic billiards the confocal spherical ellipse $e^{(j)}$ coincides with $e^{(N-2-j)}$ w.r.t.\ the opposite caustic.
\\
We obtain other families of incircles when we focus on pairs of consecutive sides of the billiards in $e^{(1)}$, $e^{(2)}$ and so on.
However, these circles are not mutually disjoint.
By the way, the centers $R_i^{(j)}$ of all these circles are the poles of diagonals of $\dots P_1 P_2 P_3 \dots$ w.r.t.\ the ellipse $e$.
\end{rem}

For the sake of completeness, we express below in \eqref{eq:w} the distance $w$ in terms of the semiaxes of $e$ and $c\,$.
For this purpose, we compute first the semiaxes $a_{e|1}$ an  $b_{e|1}$ of $e^{(1)}$, since we need the coordinates of $S_2^{(1)}$. 
From \eqref{eq:R} and \eqref{eq:Rt} follows
\[  R_2 = \left(\frac{a_e^2\cos t_2'}{a_c}, \ \frac{b_e^2\sin t_2'}{b_c}\right), 
    \quad \ol{Q_2R_2} = \frac{k_e \Vert\Vkt t_c(t_2')\Vert}{a_c b_c}\,.
\]
By virtue of \Thmref{Thm_1}, the tangents from $R_2$ to the confocal hyperbola through $Q_2$ contact at
\[  Q_2 = (a_c\cos t_2',\ b_c\sin t_2') \in c \quad\mbox{and}\quad
    S_2^{(1)} = (a_{e|1}\cos t_2',\ b_{e|1}\sin t_2') \in e^{(1)}.
\]
Both points lie on the polar of $R_2$ w.r.t.\ the hyperbola in question with the elliptic coordinate $k_h = -\Vert\Vkt t_e(t_2')\Vert^2$.
This yields the condition
\[  \frac{a_e^2\cos t_2'}{a_c(a_c^2 + k_h(t_2'))}\,a_{e|1}\cos t_2' +
    \frac{b_e^2\sin t_2'}{b_c(b_c^2 + k_h(t_2'))}\,b_{e|1}\sin t_2' = 1
\]    
or, by virtue of \eqref{eq:k_h},
\[  a_e^2 b_c a_{e|1} - b_e^2 a_c b_{e|1} = a_c b_c d^2, \quad\mbox{where}\quad
    a_{e|1}^2 - b_{e|1}^2 = d^2.
\]
We eliminate $a_{e|1}$ and obtain after some computation the quadratic equation
\[  (b_e^4 - 2b_c^2 b_e^2 - b_c^2 d^2)\,b_{e|1}^2 + 2a_c^2 b_c b_e^2\,b_{e|1}
    + b_c^2(a_c^2 d^2 - a_e^4) = 0.
\]     
The second solution besides $b_{e|1} = b_c$ is
\begin{equation}\label{eq:be1}
  b_{e|1} 
    = \frac{b_c(a_c^2 d^2 - a_e^4)}{b_e^4 - 2b_c^2 b_e^2 - b_c^2 d^2}
    = \frac{b_c(a_e^2 b_e^2 + d^2 k_e)}{a_c^2 b_c^2 - k_e^2}.
\end{equation}
This implies 
\begin{equation}\label{eq:ae1}
   a_{e|1} = \frac{a_c}{a_e^2 b_c}\left(b_c d^2 + b_e^2 b_{e|1}\right )
    = \frac{a_c(a_e^2 b_e^2 - d^2 k_e)}{a_c^2 b_c^2 - k_e^2}
\end{equation}
and
\begin{equation}\label{eq:ke1}
  k_{e|1} = b_{e|1}^2 - b_c^2 
  = k_e \left( \frac{2a_c b_c a_e b_e}{a_c^2 b_c^2 - k_e^2} \right)^{\!\!2}. 
\end{equation}
This leads finally to
\begin{equation}\label{eq:w}
  w:= \frac{2 a_e b_e \sqrt{k_e^3}}{a_c^2 b_c^2 - k_e^2}\,.
\end{equation}

Negative semiaxes $a_{e|1}, b_{e|1}$ and a negative $w$ in the formulas above mean that the points $S_i^{(1)}$ are located on the respective second branches of the hyperbolas and the incircles of the triangles $P_i P_{i+1} S_i^{(1)}$ become  excircles.
In the case of a vanishing denominator for $k_e = a_c b_c$ (periodic four-sided billiard) the ellipse $e^{(1)}$ is the line at infinity. 

If on the right-hand side of the formulas \eqref{eq:ae1}, \eqref{eq:be1} and \eqref{eq:ke1} we replace $a_e, b_e, k_e$ respectively by $a_{e|1}, b_{e|1}, k_{e|1}$, then we obtain expressions for $a_{e|3}, b_{e|3}, k_{e|3}$, i.e.,
\begin{equation}\label{eq:ke3}
  k_{e|3} = k_{e|1} \left( \frac{2a_c b_c a_{e|1} b_{e|1}}{a_c^2 b_c^2 - k_{e|1}^2} \right)^{\!\!2}, 
\end{equation}
and this can be iterated.

\subsection{Conjugate billiards}

For two confocal ellipses $c$ and $e$, there exists an axial scaling
\begin{equation}\label{eq:alpha}
  \alpha\!:\,(x,y)\mapsto \Bigl(\smFrac{a_e}{a_c} x, \, \smFrac{b_e}{b_c} y\Bigr) 
    \quad \mbox{with} \quad c\to e\,.
\end{equation}
Corresponding points share the parameter $t$. 
Hence, they belong to the same confocal hyperbola (\Figref{fig:Affinitaet2}). 
The affine transformation $\alpha$ maps the tangency point $Q_i\in c$ of the side $P_iP_{i+1}$ to a point $P_i'\in e$, while $\alpha^{-1}$ maps $P_i$ to the tangency point $Q_{i-1}'$ of $P_{i-1}'P_i'$, i.e.,
\[  \alpha\!: \, Q_i\mapsto P_i',\quad Q_{i-1}'\mapsto P_i\,.
\]
This results from the symmetry between $t_i$ and $t_i'$ in the equation 
\begin{equation}\label{eq:P in tQ}
  b_c a_e\cos t_i\cos t_i' + a_c b_e\sin t_i\sin t_i' = a_c b_c
\end{equation}
which expresses that $P_i\in e$ with parameter $t_i$ lies on the tangent to $c$ at $Q_i$ with parameter $t_i'$.
Referring to \Figref{fig:Affinitaet2}, $\alpha$ sends the tangent $[P_{i-1}',P_i']$ to $c$ at $Q_{i-1}'$ to the tangent $[R_{i-1},R_i]$ to $e$ at $P_i\,$.
Hence, by $\alpha$ the polygon $Q_1Q_2\dots$ is mapped to $P_1'P_2'\dots$ and futhermore to that of the poles $R_1R_2\dots$ of the billiard's sides $P_1P_2,\,P_2P_3,\dots\,$.

\begin{figure}[htb] 
  \centering 
  \includegraphics[width=90mm]{\pfad 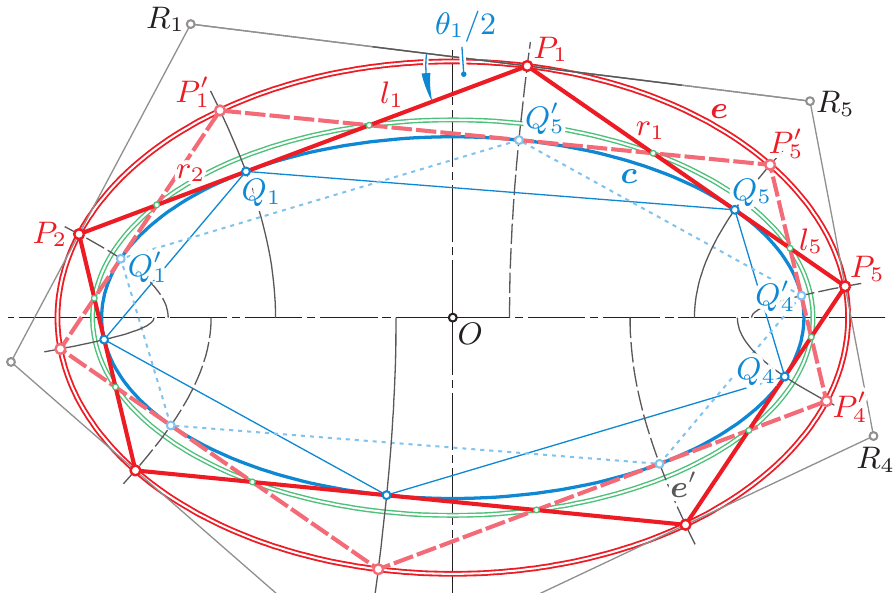}
  \caption{The periodic billiard $P_1 P_2 \dots P_5$ in $e$ with the caustic $c$ and the conjugate billiard $P_1' P_2' \dots P_5'\,$.}
  \label{fig:Affinitaet2}
\end{figure}

\begin{dft}\label{def:conjugate}
Referring to \Figref{fig:Affinitaet2}, the billiard $\dots P_0'P_1'P_2'\dots$ is called {\em conjugate} to the billiard $\dots P_0P_1P_2\dots$ in the ellipse $e$ with the ellipse $c$ as caustic, when the 
axial scaling $\alpha\!:\,c\to e$ defined in \eqref{eq:alpha} maps the tangency point $Q_i$ of the side $P_iP_{i+1}$ to the vertex $P_i'$. 
\end{dft}

\begin{lem}\label{lem:conjugate_ell}
For each billiard $\dots P_0P_1P_2\dots$ in the ellipse $e$ with the ellipse $c$ as caustic, there exists 
a unique conjugate billiard $\dots P_0'P_1'P_2'\dots$, and the relation between the two billiards in $e$  is symmetric.
Moreover, 
\begin{equation}\label{eq:Ivory} 
   l_i = \ol{P_iQ_i} = \ol{P_i'Q_{i-1}'} = r_i' \quad \mbox{and} \quad
   r_i = \ol{P_iQ_{i-1}} = \ol{P_{i-1}'Q_{i-1}'} = l_{i-1}'.
\end{equation}
\end{lem}

\begin{proof} From the symmetry in \eqref{eq:P in tQ} follows for $\alpha\!:\,c\to e$ that $P_i$ is the preimage of $P_i$ is the tangency point $Q_{i-1}$ of $P_{i-1}'P_i'$.
The congruences stated in \eqref{eq:Ivory} follow from Ivory's Theorem for the two diagonals in the curvilinear quadrangle $P_iP_i'Q_iQ_{i-1}'$.
In view of the sequence of parameters $t_1,t_1',t_2,t_2',t_3,\dots$ of the vertices $P_1$, $P_1'$, $P_2$, $P_2'$, $P_3,\dots$ on $e$, the switch between the original billiard and its conjugate corresponds the interchange of $t_i$ with $t_i'$ for $i=1,2,\dots$\,.
\end{proof}

\medskip
Finally we recall that, based on the Arnold-Liouville theorem from the theory
of completely integrable systems, it is proved in \cite{Izmestiev} that there exist {\em canonical coordinates} $u$ on the ellipses $e$ and $c$ such that for any billiard the transitions from $P_i \to P_{i+1}$ and  $Q_i\to Q_{i+1}$ correspond to shifts of the respective canonical coordinates $u_i$ and $u_{i+1}$ by $2\mskip 1mu\varDelta u$.
Explicit formulas for the parameter transformation $t\mapsto u$ are provided in \cite{Sta_II}.

\begin{figure}[htb] 
  \centering 
  \includegraphics[width=88mm]{\pfad 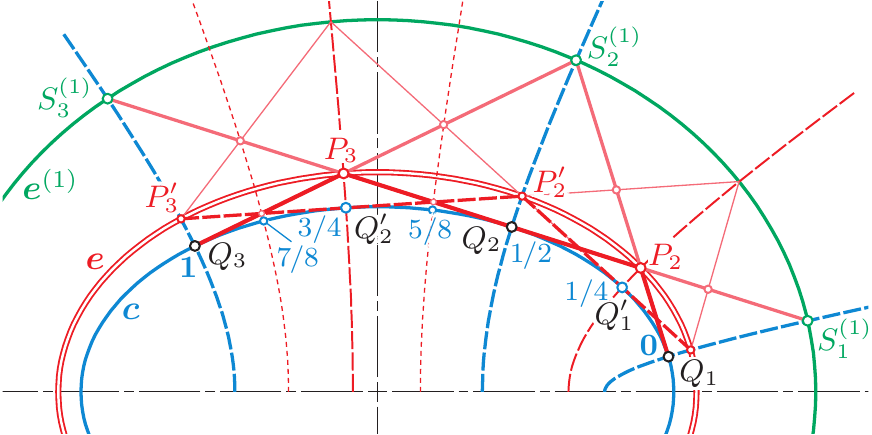} 
  \caption{An example of canonical coordinates on $c$, $e$ and $e^{(1)}$, this time with origin $Q_1$ and unit point $Q_3$}
  \label{fig:tangente2}
\end{figure}

\Figref{fig:tangente2} shows how on $c$ such coordinates can be constructed by iterated subdivision, provided that $Q_N$ and $Q_2$ get the respective canonical coordinates $u=0$ and $1$.
A comparison with \Figref{fig:Affinitaet2} reveals that, in the sense of a canonical parametrization, the contact point $Q_i$ is exactly halfway from the $P_i$ to $P_{i+1}$, i.e., 
\begin{equation}\label{eq:canon_coo}
  u_i' = u_i + \varDelta u, \quad u_{i+1} = u_i + 2\mskip 1mu\varDelta u\,.
\end{equation}
Hence, the transition from a billiard to its conjugate is equivalent to a shift of canonical coordinates by $\varDelta u$.

\subsection{Billiards with a hyperbola as caustic}
As illustrated in \Figref{fig:hyp_Kaust}, billiards in ellipses $e$ with a confocal hyperbola $c$ as caustic are zig-zags between an upper and lower subarc of $e$.
If the initial point $P_1$ is chosen at any point of intersection between $e$ and the hyperbola $c$, then the billiard is twofold covered, and the first side $P_1P_2$ is tangent to $c$ at $P_1$. 

Here we report briefly, in which way these billiard differ from those with an elliptic caustic.
Proofs are left to the readers. 
In view of the associated Poncelet grid, we start with the analogue to \Thmref{thm:symmetry} (see Figures~\ref{fig:hyp_Kaust}, \ref{fig:hyp_Kaust3} and \ref{fig:hyp_Kaust2}).

\begin{figure}[hbt] 
  \centering 
  \includegraphics[width=115mm]{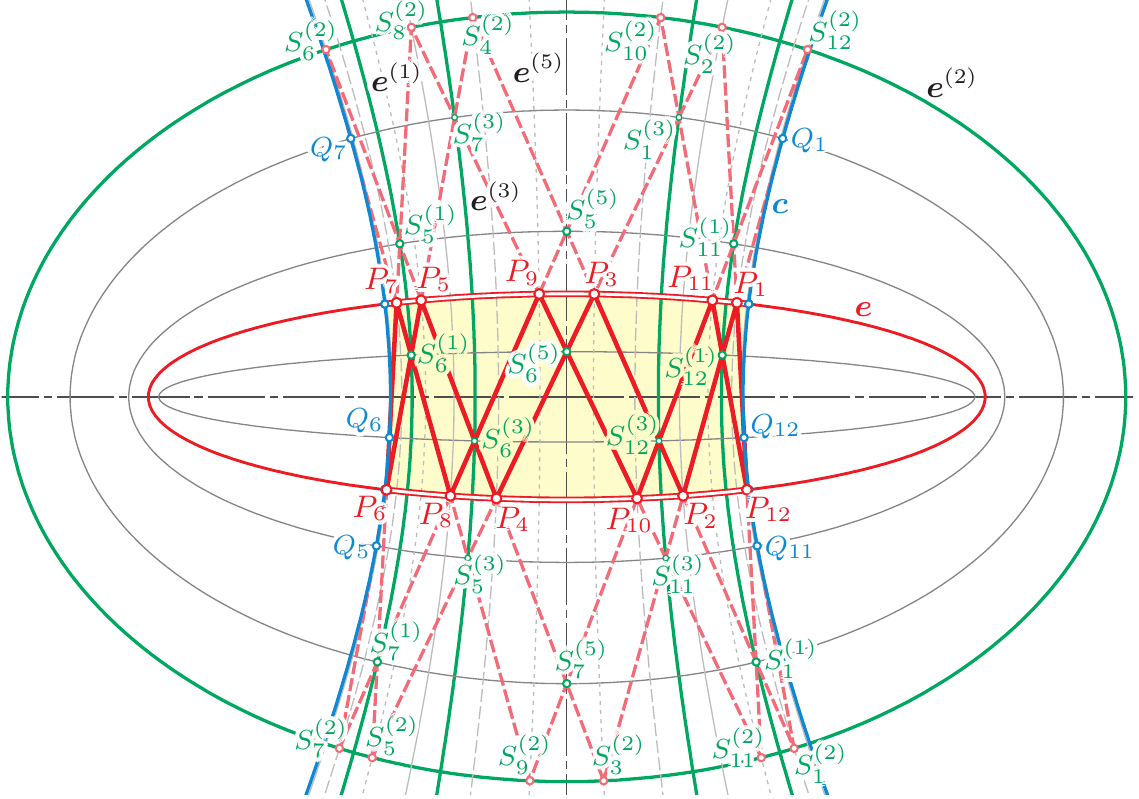} 
  \caption{Periodic billiard $P_1P_2\dots P_{12}$ in the ellipse $e$ with the hyperbola $c$ as caustic, together with the hyperbolas $e^{(1)}$, $e^{(3)}$ and the ellipse $e^{(2)}$ with the inscribed billiard consisting of three quadrangles
$S_i^{(2)} S_{i+3}^{(2)} S_{i+6}^{(2)} S_{i+9}^{(2)}$.}
  \label{fig:hyp_Kaust}
\end{figure}

\begin{thm}\label{thm:symmetry_hyp} 
Let $\dots P_0 P_1 P_2 \dots$ be a billiard in the ellipse $e$ with the hyperbola $c$ as caustic. 
\begin{enumerate}
\item Then the points $S_i^{(1)}$, $S_i^{(3)}, \dots$ are located on confocal ellipses through the contact point $Q_i$ of $[P_i,P_{i+1}]$ with $c$, while the points
$S_i^{(2)}$, $S_i^{(4)},\dots$ are located on the confocal hyperbola through $P_i$.
\item  
For even $j$, the points $\dots S_i^{(j)} S_{i+(j+1)}^{(j)}S_{i+2(j+1)}^{(j)}\dots$ are vertices of another billiard with the caustic $c$ inscribed in a confocal ellipse $e^{(j)}$, provided that the $S_i^{(j)}$ are finite.
\\
For odd $j$, the points $S_i^{(j)}$ are located on confocal hyperbolas $e^{(j)}$ or an axis of symmetry. 
At each vertex of $\dots S_i^{(j)} S_{i+(j+1)}^{(j)}S_{i+2(j+1)}^{(j)}\dots$, one angle bisector is tangent to $e^{(j)}$. 
\\
All conics $e^{(j)}$ are independent of the position of the initial vertex $P_1\in e$. 
\end{enumerate}
\end{thm}

At the 12-periodic billiard depicted in \Figref{fig:hyp_Kaust}, the billiard inscribed to $e^{(2)}$ splits into three quadrangles (dashed).
Note that for odd $j$ there are some points $S_i^{(j)}$ where the tangent to $e^{(j)}$ is the interior bisector of the angle $\angle\, S_{i-(j+1)}^{(j)} S_i^{(j)} S_{i+(j+1)}^{(j)}$.
Hence, we obtain no billiards inscribed to hyperbolas $e^{(j)}$ with the hyperbola $c$ as caustic.
In \Figref{fig:hyp_Kaust2}, the 6-periodic billiard $P_1 P_2\dots P_6$ yields two triangles $S_i^{(1)}S_{i+2}^{(1)}S_{i+4}^{(1)}$ inscribed to $e^{(1)}$; one of them is shaded in green.  

\medskip
\Lemref{lem:w=konst} is also valid for hyperbolas as caustic.
An example is depicted in \Figref{fig:hyp_Kaust3}: 
The power $w^2$ of $P_9$ w.r.t.\ the circle tangent to the four consecutive sides $[P_7,P_8]$, $[P_8,P_9]$,$[P_9,P_{10}]$, and $[P_{10},P_{11}]$ equals that of $P_{12}$ w.r.t.\ the incircle of the quadrilateral with sides $[P_{10},P_{11}]$,  $[P_{11},P_{12}]$, $[P_{12},P_1]$, and $[P_1,P_2]$.

\begin{figure}[t] 
  \centering 
 \includegraphics[width=80mm]{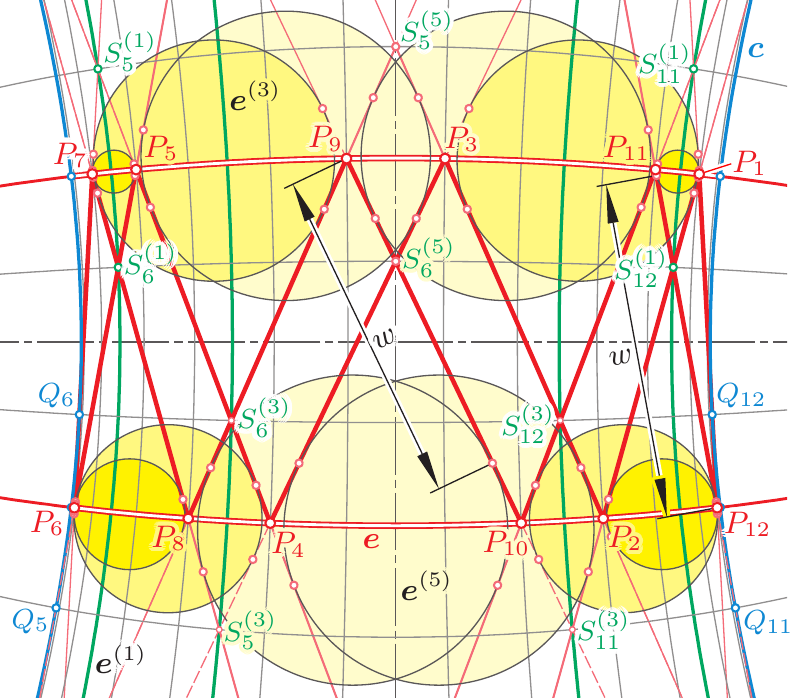} 
  \caption{The power $w^2$ of $P_9$ w.r.t.\ the circle tangent to the sides $[P_7,P_8]$, $[P_8,P_9]$,$[P_9,P_{10}]$ equals that of $P_{12}$ w.r.t.\ the circle tangent to  $[P_{10},P_{11}]$, $[P_{11},P_{12}]$, $[P_{12},P_1]$.}
  \label{fig:hyp_Kaust3}
\end{figure}

\medskip
Also for billiards $P_1P_2\dots$ in $e$ with a hyperbola $c$ as caustic, there exists a conjugate billiard $P_1'P_2'\dots$, and the relation is symmetric.
However, the definition is different.
It uses the singular affine transformation
\begin{equation}\label{eq:alpha_h}
   \alpha_h\!:\,e\to F_1F_2 \zwi\mbox{with}\zwi P_i\mapsto T_i=[P_i,P_{i+1}]\cap[F_1,F_2],
\end{equation}
with $F_1$ and $F_2$ as the focal points of $e$ and $c$.

\begin{figure}[t] 
  \centering 
  \includegraphics[width=65mm]{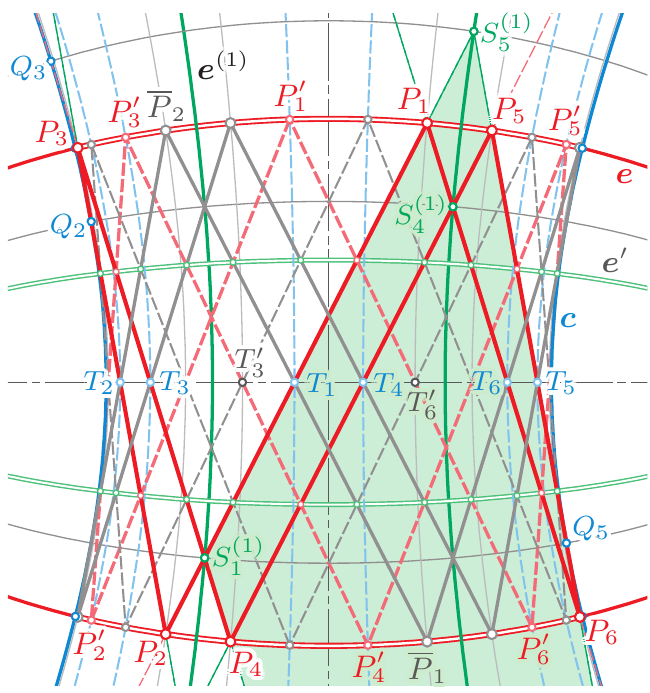} 
  \caption{Periodic billiard $P_1 P \dots P_6$ in the ellipse $e$ with the hyperbola $c$ as caustic, together with the conjugate billiard $P_1' P_2' \dots P_6'$.
  The associated polygon with vertices on $e^{(1)}$ splits into two triangles
$S_1^{(1)}S_3^{(1)}S_5^{(1)}$ (green shaded) and $S_2^{(1)}S_4^{(1)}S_6^{(1)}$.}
  \label{fig:hyp_Kaust2}
\end{figure}

\begin{dft}\label{def:conjugate_hyp}
Referring to \Figref{fig:hyp_Kaust2}, the billiard $\dots P_0'P_1'P_2'\dots$ in the ellipse $e$ with the hyperbola $c$ as caustic is called {\em conjugate} to the billiard $\dots P_0P_1P_2\dots$ in $e$ with the same caustic $c$ if the axial scaling $\alpha_h$ defined in \eqref{eq:alpha_h} maps the point $P_i'$ to the intersection $T_i$ of $P_iP_{i+1}$ with the principal axis. 
\end{dft}

\begin{lem}\label{lem:conjugate_hyp}
Let $\dots P_0P_1P_2\dots$ be a billiard in the ellipse $e$ with the hyperbola $c$ as caustic.
Then to this billiard and to its mirror w.r.t.\ the principal axis exists a conjugate billiard $\dots P_0'P_1'P_2'\dots$, and it is unique up to a reflection in the principal axis.
The relation between the two conjugate billiards in $e$ is symmetric.
Moreover, if $T_i'$ denotes the intersection of $P_i'P_{i+1}'$ with the principal axis, then
\begin{equation}\label{eq:Ivory_hyp} 
   \ol{P_iT_i} = \ol{P_i'T_{i-1}'} \zwi \mbox{and} \zwi
   \ol{P_iT_{i-1}} = \ol{P_{i-1}'T_{i-1}'}.
\end{equation}
\end{lem}

\begin{proof}
The singular affine transformation $\alpha_h$ maps $P_1'$ to $T_1$ and $P_1$ to a point $T'$ (=\,$T_6'$ in \Figref{fig:hyp_Kaust2}).
We assume that $P_1$ and $P_1'$ lie on the same side of the principal axis, since otherwise we apply a reflection in the axis. 
Then we obtain a curvilinear Ivory quadrangle $P_1' T_1 T' P_1$ with diagonals of equal lengths.
On the other hand, the lines $[P_1', T']$ and $[P_1,T_1]$ must contact the same confocal conic (see, e.g., \cite[p.~153]{Boehm2} or \cite[Lemma~1]{MonGeom}). 
Hence, the billiard through the point $P_1'\in e$ with caustic $c$ contains one side on the line $[P_1',T']$.
Iteration confirms the claim. 
\end{proof}

Let $\ol P_1$ and $\ol P_2$ be the images of $P_1$ and $P_2$ under reflection in the principal axis of $e$ (\Figref{fig:hyp_Kaust2}). 
Then, a comparison with \Figref{fig:tangente2} reveals that the confocal hyperbola through the intersection $T_1 = [P_1,P_2]\cap[\ol P_1,\ol P_1]$ lies `in the middle' between the hyperbolas through $P_1$ and $P_2$.

\begin{rem}
If $\dots P_0P_1P_2\dots$ and $\dots P_0'P_1'P_2'\dots$ is a pair of conjugate billiards in an ellipse $e$, then the points of intersection $[P_i,P_{i+1}]\cap[P_i',P_{i-1}']$ are located on a confocal ellipse $e'$ inside $e$.
This holds for ellipses and hyperbolas as caustics.
If the billiards are $N$-periodic, then in the elliptical case, the restriction of the two billiards to the interior of $e'$ is $2N$-periodic; conversely, $e$ plays the role of $e^{(1)}$ w.r.t.\ $e'$ (\Figref{fig:Affinitaet2}).
In the hyperbolic case, the restriction to the interior of $e'$ gives two symmetric $2N$-periodic billiards, provided that also the reflected billiards are involved (\Figref{fig:hyp_Kaust2}).
\end{rem}

\medskip
We conclude with citing a result from \cite{Sta_III} about billiards in ellipses.
It states that for each billiard $\dots P_1P_2P_3\dots$ in $e$ with a hyperbola as caustic there exists a billiard $\dots P_1^\ast P_2^\ast P_3^\ast \dots$ in $e^\ast$ with an ellipse as caustic such that corresponding sides $P_iP_{i+1}$ and $P_i^\ast P_{i+1}^\ast$ are congruent.

\section{Periodic $\boldsymbol{N}$-sided billiards} 

Let the billiard $P_1P_2\dots P_N$ in the ellipse $e$ be periodic with an ellipse $c$ as caustic.
Then, the sequence of parameters $t_1, t_1', t_2, \dots, t_N, t_N'$ of the vertices $P_i$ and the intermediate contact points $Q_1,\dots,Q_N$ with $c$ is cyclic.
Each side line intersects only a finite number of other side lines.
Hence, the corresponding Poncelet grid contains a finite number of confocal ellipses $e^{(j)}$  through the points  $S_i^{(j)}$, namely $[\frac {N-2}2]$ (including possibly the line at infinity), provided that $N\ge 5$ (\Figref{fig:Poncelet_grid}).
The sequence of ellipses $e, e^{(1)}, e^{(2)}, \dots$ is cyclic, and 
\begin{equation}\label{eq:e^i=e^j}
  e^{(j)} = e^{(N-2-j)}.
\end{equation}
For example, in the case $N=7$ (\Figref{fig:Inkreise}), the ellipse $e^{(2)}$ coincides with $e^{(3)}$.

\begin{dft}\label{def:turning_no}
The sum of exterior angles $\theta_i$ of a periodic billiard in an ellipse $e$ is an integer
multiple of $2\pi$, namely $2\tau\pi$. 
We call $\tau\in\mathbb N$ the {\em turning number} of the billiard.
It counts the loops of the billiard around the center $O$ of $e$, anti-clockwise or clockwise.
\end{dft}

If the periodic billiard $P_1 P_2\dots P_N$ has the turning number $\tau = 1$ (\Figref{fig:Poncelet_grid2}), then the billiard $S_1^{(1)} S_3^{(1)} S_5^{(1)}\dots$ in $e^{(1)}$ has $\tau = 2$, that of $S_1^{(2)} S_4^{(2)} \dots$ in $e^{(2)}$ the turning number $\tau = 3$, and so on. 
In cases with $g = \mathrm{gcd}(N, \tau) > 1$ the corresponding billiard splits into $g$ $\frac Ng$-sided billiards, each with turning number $\tau/g$ (note \cite[Theorem~1.1]{Schwartz}).

\subsection{Symmetries of periodic billiards}

The following is a corollary to \Thmref{thm:symmetry}.

\begin{cor}\label{cor:symmetry} 
Let $P_1P_2\dots P_N$ be an $N$-sided periodic billiard in the ellipse $e$ with the ellipse $c$ as caustic. 
\renewcommand\labelenumi{{(\roman{enumi})}}
\begin{enumerate}
\item 
 For even $N$ and odd $\tau$, the billiard is centrally symmetric.
\item 
 For odd $N = 2n+1$ and odd $\tau$, the billiard is centrally symmetric to the conjugate billiard, where $P_i$ corresponds to $P_{i+n}'$.\footnotemark
\item 
 If $N$ is odd and $\tau$ is even, then the conjugate billiard coincides with the original one, and $P_i = P'_{i+n}$.
\end{enumerate}
\end{cor}
\footnotetext{
All subscripts in this section are understood modulo $N$.}

\begin{proof}
By virtue of \Thmref{thm:symmetry}, the lines $[P_{i-j-1},P_{i-j},]$ and $[P_{i+j},P_{i+j+1}]$ for $j = 1,2,\dots$ meet at the point $S_i{(j)}$ on the confocal hyperbola through $P_i$.
\\[1.0mm]
(i) This means for even $N = 2n$, odd $\tau$ and $j = n-1$, that also the opposite vertex $S_i^{(n)} = P_{i-n} = P_{i+n}$ belongs to this hyperbola. 
If $P_i$ is specified at a vertex on the minor axis of the ellipse $e$, then $P_{i+n}$ is the opposite vertex.
Continuity implies that the two points belong to different branches of the hyperbola and are symmetric w.r.t.\ the center $O$ of $e$. 
\\[1.0mm]
(ii), (iii): If $N$ is odd, say $N = 2n+1$, then for $j = n-1$ the sides $[P_{i-n+1},P_{i-n}]$ and $[P_{i+n-1},P_{i+n}]$ intersect at a point on the hyperbola through $Q_{i+n}$ and $P'_{i+n}$.
For odd $\tau$ (Figures~\ref{fig:Affinitaet2} and \ref{fig:Poncelet_grid2}), the same continuity argument as before proves that $P_i$ and $P'_{i+n}$ are opposite w.r.t.\ $O$.  
\\
If $\tau$ is even (\Figref{fig:Inkreise2}), then the choice of $P_i\in e$ on an axis of symmetry shows the coincidence with $P'_{i+n}\in e$, and this must be preserved, when $P_i$ varies continuously on $e$.
In the case of even $\tau$ and $N$ the billiard splits.
\end{proof}

\begin{figure}[hbt] 
  \centering 
  \includegraphics[width=95mm]{\pfad 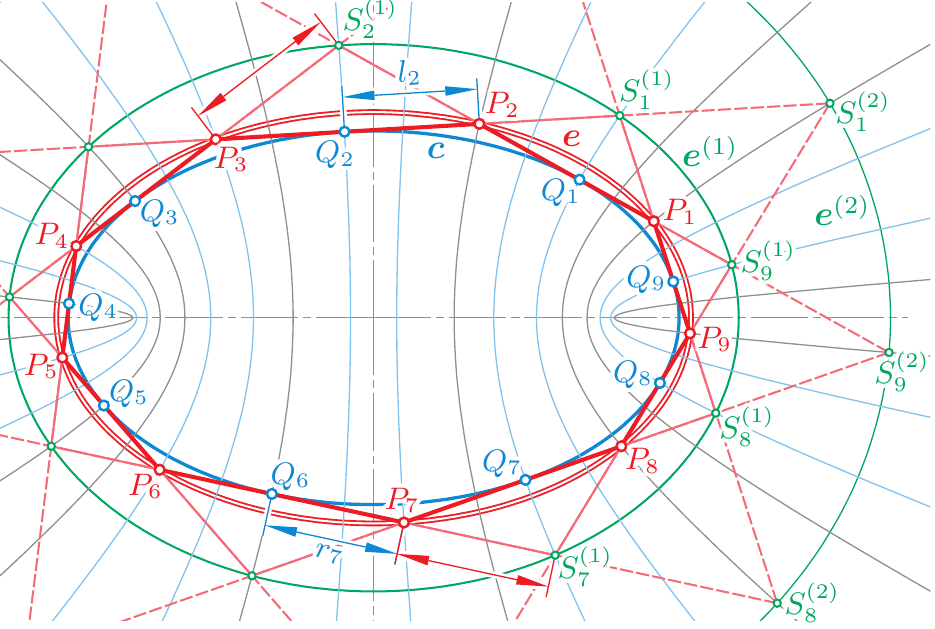} 
  \caption{Periodic billiard $P_1P_2\dots P_9$ with $\tau = 1$. 
  Note that $l_2 = \ol{P_2Q_2} = \ol{Q_6P_7} = r_7$ and $\ol{S_2^{(1)}P_3} = \ol{P_7 S_7^{(1)}}$. 
  The associated billiard in $e^{(2)}$ 
  splits into three triangles.} 
  \label{fig:Poncelet_grid2}
\end{figure}

\medskip
The billiards with a hyperbola $c$ as caustic (see Figures~\ref{fig:hyp_Kaust} and \ref{fig:hyp_Kaust2}) oscillate between the upper and lower section of $e$.
Therefore, only billiards with an even $N$ can be periodic.
Also for billiards of this type, it possible to define a turning number $\tau$ which counts how often the points $P_1, \ol P_2, P_3,\dots, P_N$ (\Figref{fig:hyp_Kaust2}) run to and fro along the upper component of $e$.\footnote{
The turning number of hyperbolic billiards becomes more intuitive when the billiard is seen as the limit of a focal billiard in the sense of \cite[Theorem~2]{Sta_III}.}
The symmetry properties of these periodic $N$-sided billiards differ from those in \Corref{cor:symmetry}.
They follow from \Thmref{thm:symmetry}, since opposite vertices $P_i$ and $P_{i+N/2}$ belong to the same confocal hyperbola.

\begin{cor}\label{cor:symmetry_hyp} 
Let $P_1P_2\dots P_N$ be an $N$-sided periodic billiard in the ellipse $e$ with the hyperbola $c$ as caustic. 
\renewcommand\labelenumi{{(\roman{enumi})}}
\begin{enumerate}
\item 
 For $N\equiv 0\pmod 4$, the billiard is symmetric w.r.t.\ the secondary axis of $e$ and $c$.
\item 
 For $N\equiv 2\pmod 4$ and odd turning number $\tau$, the billiards are centrally symmetric.
 For even $\tau$, each billiard is symmetric w.r.t.\ the principal axis of $e$ and $c$. 
\end{enumerate}
\end{cor}

\subsection{Some invariants}

As a direct consequence of the results so far, we present new proofs for the invariants k101, k118 and k119 listed in \cite[Table~2]{80}, though this table refers already to proofs for some of them in \cite{Ako-Tab} and \cite{Chavez}.  

We begin with a result that has first been proved for a much more general setting in \cite[p.~103]{Tabach}. 

\begin{lem}
The length $L_e$ of a periodic $N$-sided billiard in the ellipse $e$ with the ellipse $c$ as caustic is independent of the position of the initial vertex $P_1\in e$.
\end{lem}

\begin{proof} 
We refer to Graves's construction \cite[p.~47]{Conics}.
According to \eqref{eq:De} holds 
\[   D_e:= \ol{Q_{i-1}P_i} + \ol{P_iQ_i} - \bow{Q_{i-1}Q_i}\,.
\]
This yields for an $N$-sided billiard with turning number $\tau$ the total length
\begin{equation}\label{eq:L}
  L_e = N\cdot D_e - \tau\cdot P_c\,,
\end{equation}
where $P_c$ denotes the perimeter of the caustic $c$.
Thus, $L_e$ does not depend of the choice of the initial vertex $P_1\in e$.
\end{proof}

If the billiard in $e$ has the turning number $\tau$, then its extension in $e^{(1)}$ has the turning number $2\tau$, and from \eqref{eq:De1,2} and \eqref{eq:w} follows
\begin{equation}\label{eq:Le1}
  L_{e|1} = N D_{e|1} - 2\tau P_c = 2N(D_e + w) - 2\tau P_c 
    = 2L_e + 2N\,\frac{2 a_e b_e \sqrt{k_e^3}}{a_c^2 b_c^2 - k_e^2}\,.
\end{equation}

\begin{figure}[htb] 
  \centering 
  \includegraphics[width=85mm]{\pfad 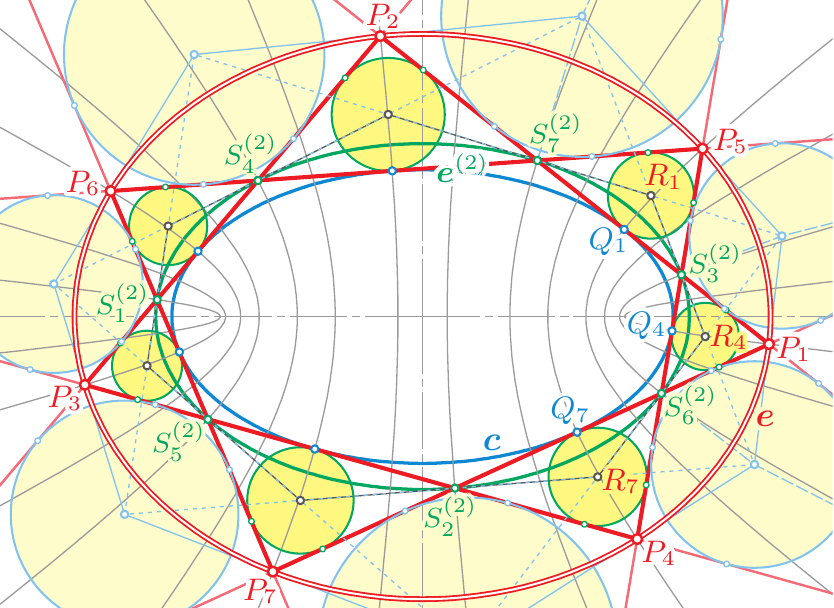} 
  \caption{Periodic billiard with $N = 7$ and $\tau=2$.}
  \label{fig:Inkreise2}
\end{figure}


The following theorem on the invariant k118 in \cite{80} deals with the lengths $r_i$ and $l_i$ of the segments $Q_{i-1}P_i$ and $P_iQ_i$, as defined in \eqref{eq:def_rili}.

\begin{thm}\label{thm:sum_ri} 
In each $N$-sided periodic billiard opposite segments are congruent, i.e., if $N=2n$, then $r_{i+n}=r_i$ and $l_{i+n}= l_i\,$, and if $N=2n+1$, then $r_{i+n} = l_{i-1}$ and $l_{i+n}= r_i\,$. 
Thus, for odd $N$ holds
\[  \sum_{i=1}^N l_i = \sum_{i=1}^N r_i = \frac{L_e}2\,. 
\]
\end{thm}

\begin{proof}
By \Corref{cor:symmetry}, for even $N = 2n$ the central symmetry implies for opposite segments $r_i = r_{i+n}$ and $l_i = l_{i+n}$.
\\
If $N = 2n+1$, then $l_i = \ol{P_iQ_i}$ shows up as $l_{i+n}'= \ol{P_{i+n}'Q_{i+n}'}$ at the conjugate billiard and, by virtue of \eqref{eq:Ivory},
this equals $r_{i+n+1} = \ol{P_{i+n+1}Q_{i+n}}$ (\Figref{fig:Poncelet_grid2}).
Similarly follows $r_i = l_{i+n}$.
\end{proof}

\begin{rem}
In the particular case $N = 3$ the two segments adjacent to any side are congruent (note in \Figref{fig:Poncelet_grid2} the triangular billiards in $e^{(2)}$).
Therefore, the Cevians $[P_i,Q_{i+1}]$ are concurrent and meet at the Nagel point of the triangle. 
This has already been proved in \cite{Reznik_generic} and agrees with the circles through $Q_i$ and centered at $R_i$ (see Figures~\ref{fig:Inkreise} and \ref{fig:Inkreise2}), which for $N=3$ are excircles of the triangle $P_1P_2P_3\,$.
\end{rem}

The following theorem has first been proved in \cite[p.~4]{Ako-Tab}.
Another proof can be found in \cite[Cor.~3.2]{Bialy-Tab}.
We give below a new proof.

\begin{thm}\label{thm:Ako-Tab} 
For the exterior angles $\theta_1,\dots,\theta_N$ of the periodic $N$-sided elliptic billiard in the ellipse $e\,$, the sum of cosines is independent of the initial vertex, namely
\[  \sum_{i=1}^N \cos\theta_i = N - J_e L_e = N - \frac{\sqrt{k_e}}{a_eb_e}\,L_e\,,
\]
where $L_e$ is the common perimeter of these billiards in $e\,$.
\end{thm}

\begin{figure}[htb] 
  \centering 
  \includegraphics[width=80mm]{\pfad 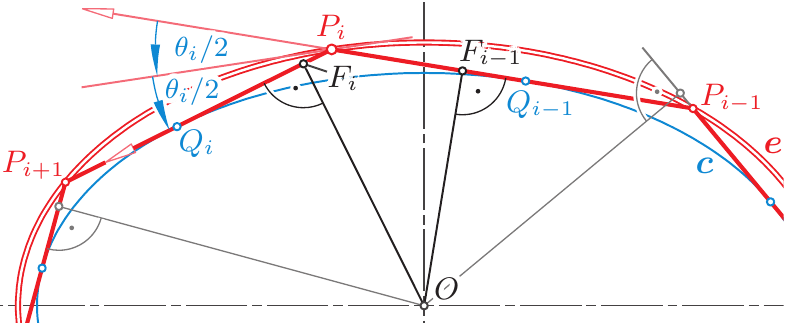} 
  \caption{Billiard $P_1P_2\dots $ with pedal points w.r.t.\ $O$.} 
  \label{fig:pedal}
\end{figure}

\begin{proof}
The pedal points $F_i$ and $F_{i-1}$ on the sides $P_iP_{i+1}$ and $P_{i-1}P_i$ w.r.t.\ the center $O$ have the position vectors 
\[ \begin{array}{c}
    \Vkt f_{i,i-1} = \Vkt p + \smFrac{\lambda_{i,i-1}}{\Vert\Vkt t_e\Vert}\left(
    \cos\smFrac{\theta_i}2\,\Vkt t \pm \sin\smFrac{\theta_i}2\,\Vkt t^\perp\right),
    \\[2.0mm]
    \mbox{where} \quad  0 = \Bigl\langle \Vkt f_{i,i-1},\,\Bigl(
     \cos\smFrac{\theta_i}2\,\Vkt t \pm \sin\smFrac{\theta_i}2\,
     \Vkt t^\perp\Bigr)\Bigr\rangle.
  \end{array}
\]
Here, $\lambda_i$ and $\lambda_{i+1}$ denote the signed distances from the vertex $P_i$ in one case towards $P_{i+1}$, in the other opposite to $P_{i-1}$.
From 
\[  \langle \Vkt p_i,\,\Vkt t\rangle = (-a_e^2 + b_e^2)\cos t\sin t 
    \zwi \mbox{and} \zwi  \langle \Vkt p_i,\,\Vkt t^\perp\rangle = -a_e b_e
\]
follows
\[  \lambda_{i,i-1} = \frac 1{\Vert\Vkt t_e\Vert}\left((-a_e^2 + b_e^2)
    \cos t \sin t \cos\smFrac{\theta_i}2 \pm a_e b_e \sin\smFrac{\theta_i}2\right)  
\]
This implies by virtue of \eqref{eq:Winkel/2} and after reversing the orientation for $\lambda_{i-1}$,  
\[  \lambda_i - \lambda_{i-1} = \ol{P_iF_i} + \ol{P_iF_{i-1}} = \frac{2a_eb_e}{\Vert\Vkt t_e\Vert^2}\,\sqrt{k_e}\,.
\] 
Since the sum over all signed lengths between $P_i$ and the adjacent pedal points  gives the total perimeter $L_e$ of the billiard, we obtain by \eqref{eq:Winkel}
\begin{equation}\label{eq:sum_cos}
   L_e =  \sum_{i=1}^N \left(\,\ol{P_iF_i} + \ol{P_iF_{i-1}}\,\right)
       =  \frac{a_e b_e}{\sqrt{k_e}}\,\sum_{i=1}^N \frac{2k_e}{\Vert\Vkt t_e\Vert^2}
       = \frac{a_e b_e}{\sqrt{k_e}}\,\sum_{i=1}^N (1 - \cos\theta_i),   
\end{equation}
hence
\begin{equation}\label{eq:sum 1/|t_e^2|}
   \sum_{i=1}^N  \frac 1{\Vert\Vkt t_e\Vert^2} = \frac{L_e}{2a_e b_e \sqrt{k_e}}\zwi
  \mbox{and also}\zwi
  \sum_{i=1}^N \cos\theta_i = N - \frac{\sqrt{k_e}}{a_e b_e}\,L_e\,, 
\end{equation}
as stated.
\end{proof}

\begin{rem}
Note that the result in \cite{Ako-Tab} relates to the interior angles of the billiard.
As already mentioned in \cite[Theorem~7]{Ako-Tab}, the constant sum of cosines holds also for the `extended' billiards in $e^{(j)}$, where the exterior angles are $\theta_i + \theta_{i+1} + \dots + \theta_{i+j}$ (note \Figref{fig:Poncelet_grid}).
\end{rem}

Similar to \eqref{eq:Ot}, the distance of $O$ to the tangent $t_P$ to $e$ at $P$ equals  
\begin{equation}\label{eq:OtP}
  \ol{Ot_P} = \frac{a_e b_e}{\Vert\Vkt t_e\Vert}\,.
\end{equation}
From \eqref{eq:sum 1/|t_e^2|} follows for the vertices of a billiard a result of \cite[Cor.~3.2]{Bialy-Tab}.
The invariant k119, first proved by P.\ Roitmann, deals with the curvature of $e\,$, namely 
\[  \kappa_e(t):= \frac{a_e b_e}{\Vert\Vkt t_e(t)\Vert^3}
\]
by \cite[p.~79]{Conics}.
The result referring to this and given below is again a consequence of \eqref{eq:sum 1/|t_e^2|}.

\begin{cor}\label{cor:OTP} 
The squared distances from the center $O$ to the tangents $t_{P_i}$ at the vertices $P_i$ of the periodic $N$-sided elliptic billiard in the ellipse $e$ have a constant sum, independent of the initial vertex, namely
\[  \sum_{i=1}^N \ol{Ot_P}^{\,2} =  \frac{a_e b_e}{2\sqrt{k_e}}\,L_e\,.
\]
The curvatures $\kappa_i$ of $e$ at the vertices $P_i$ give rise to an invariant sum
\[ \sum_{i=1}^N \kappa_i^{2/3} = \frac {L_e}{2\sqrt{k_e}}\,(a_eb_e)^{-1/3}\,. 
\]
\end{cor}

We conclude this section with a comment on \eqref{eq:e^i=e^j} in connection with \eqref{eq:ke1} and \eqref{eq:ke3}.
For example, we obtain $k_{e|1}=k_e$ for $N=3$, $k_e = a_c\,b_c$ for $N=4$, $k_{e|2}=k_{e|1}$ for $N=5$, and $k_{e|2}=\infty$ for $N=6$.
This yields algebraic conditions for the dimension of the ellipse $e$ with an inscribed $N$-periodic billiard when the ellipse $c$ is given as caustic.
However, we need to recall that A.\ Cayley gave already an explicit solution for this problem in  a projective setting (\cite{Griffiths} or \cite[Theorem~9.5.4]{Conics}).
Other approaches are provided in \cite[Sect.~VI]{Birkhoff} and \cite[Sect.~11.2.3.9]{Duistermaat}. 
Equivalent conditions in terms of elliptic functions can be deduced from \cite[Corollary~3]{Sta_II}. 

\subsection*{Acknowledgment}
The author is grateful to Dan Reznik and Ronaldo Garcia for inspirations and interesting discussions.



\begin{thebibliography}{1}

\bibitem{Akopyan}
 A.W.\ Akopyan, A.I.\ Bobenko: \textit{Incircular nets and confocal conics.} 
 Trans.\ Amer.\ Math.\ Soc.\ {\bf 370}  (2018), 2825--2854.

\bibitem{Ako-Tab}
{A.\ Akopyan, R.\ Schwartz, S.\ Tabachnikov:}
{\em Billiards in ellipses revisited.}
Eur.\ J.\ Math.\ 2020. \url{doi:10.1007/s40879-020-00426-9}

\bibitem{Bialy-Tab}
{M.\ Bialy, S.\ Tabachnikov:}
{\em Dan Reznik's identities and more.}
Eur.\ J.\ Math.\ 2020, \url{doi:10.1007/s40879-020-00428-7}

\bibitem{Birkhoff} 
{G.D.\ Birkhoff:} {\em Dynamical Systems.}
American Mathematical Society, Colloquium Publications, Vol.~9,
Providence/Rhode Island, 1966 

\bibitem{Bobenko} 
{A.I.\ Bobenko, W.K.\ Schief, Y.B.\ Suris, J.\ Techter:}
\textit{On a discretization of confocal quadrics. A geometric approach to general parametrizations.} 
Int.\ Math.\ Res.\ Not.\ IMRN {\bf 2020}, no.~24, 10180--10230. 

\bibitem{Bob_Fairly}  
{A.I.\ Bobenko, A.Y. Fairley:}
\textit{Nets of Lines with the Combinatorics of the Square Grid and with Touching Inscribed Conics.} 
Discrete Comput.\ Geom.\ 2021. \url{https://doi.org/10.1007/s00454-021-00277-5}

\bibitem{Boehm2} 
W.\ B\"ohm: \textit{Die Fadenkonstruktionen der Fl\"achen zweiter Ordnung.}
 Math.\ Nachr.\ {\bf 13} (1955), 151--156.

\bibitem{Boehm1} 
 W.\ B\"ohm: \textit{Ein Analogon zum Satz von Ivory.} 
 Ann.\ Mat.\ Pura Appl.\ (4) {\bf 54}  (1961), 221--225. 

\bibitem{Chasles} 
 M.\ Chasles: \textit{Propri\'et\'es g\'en\'erales des arcs d'une section conique,
 dont la difference est rectifiable.} Comptes Rendus 
 hebdomadaires de s\'eances de l'Acad\'emie des sciences
 {\bf 17}  (1843), 838--844. 

\bibitem{Chasles2} 
 M.\ Chasles: \textit{R\'esum\'e d'une th\'eorie des coniques sph\'eriques homofocales.}
Comptes Rendus des s\'eances de l'Acad\'emie des Sciences {\bf 50} (1860), 623--633. 

\bibitem{Chavez}
 A.\ Chavez-Caliz: \textit{More about areas and centers of Poncelet polygons.}  
Arnold Math.\ J.\ {\bf 7} (2021), 91--106.

\bibitem{DR_russ}
V.\ Dragovi\'c, M.\ Radnovi\'c: {\em Integrable billiards and quadrics.}
Russian Math.\ Surveys {\bf 65}/2 (2010), 319--379.

\bibitem{DR_Buch}
V.\ Dragovi\'c, M.\ Radnovi\'c: {\em Poncelet porisms and beyond. 
Integrable Billiards, Hyperelliptic Jacobians and Pencils of Quadrics.}
Birkh\"auser/Springer Basel AG, Basel 2011. 

\bibitem{Duistermaat} 
J.J.\ Duistermaat, {Discrete Integrable Systems. QRT Maps and Elliptic Surfaces.} 
Springer Monographs in Mathematics 304, Springer, New York (2010)

\bibitem{Conics}
{G.\ Glaeser, H.\ Stachel, B.\ Odehnal:} \textit{The Universe of Conics.} 
 From the ancient Greeks to 21\textsuperscript{st} century developments.
 Springer Spectrum, Berlin, Heidelberg 2016.

\bibitem{Griffiths}
{Ph.\ Griffiths, J.\ Harris:} \textit{On Cayley's explicit solution to Poncelet's porism.}
L'Enseignement Math\'ematique 24/1-2 (1978), 31--40.

\bibitem{Halbeisen}
{L.\ Halbeisen, N.\ Hungerb\"uhler:}
\textit{A Simple Proof of Poncelet's Theorem (on the occasion of its bicentennial).}
Amer.\ Math.\ Monthly {\bf 122}/6 (2015), 537--551.

\bibitem{Izmestiev}
{I.\ Izmestiev, S.\ Tabachnikov:} \textit{Ivory's Theorem revisited.}
 J.\ Integrable Syst.\ {\bf 2}/1 (2017), xyx006, 
  \url{https://doi.org/10.1093/integr/xyx006}



\bibitem{Reznik_surprise}
{D.\ Reznik, R.\ Garcia, J.\ Koiller:} 
\textit{Can the elliptic billiard still surprise us\,?}
 Math.\ Intelligencer {\bf 42} (2020), 6--17.

\bibitem{Reznik_generic}
{D.\ Reznik, R.\ Garcia:} 
\textit{The Circumbilliard: Any Triangle can be a 3-Periodic.}
\url{arXiv:2004.06776 [math.DS]}, 2020. 

\bibitem{80}
{D.\ Reznik, R.\ Garcia, J.\ Koiller:}
\textit{Eighty New Invariants of N-Periodics in the Elliptic Billiard.}
\url{arXiv:2004.12497v11 [math.DS]}, 2020.
 	
\bibitem{Schwartz}
{R.E.\ Schwartz:} {\em The Poncelet grid.}
Adv.\ Geom.\ {\bf 7} (2007), 157--175. 

\bibitem{MonGeom}
{H.\ Stachel:} \textit{Recalling Ivory's Theorem.} 
 {FME Transactions {\bf 47}, No~2 (2019), 355--359.}


\bibitem{Sta_II}
{H.\ Stachel:} {\em On the Motion of Billiards in Ellipses.} 
 \url{arXiv:2105.03624v1 [math.DG]} (2021). 


\bibitem{Sta_III}
{H.\ Stachel:} {\em Isometric Billiards in Ellipses and Focal Billiards in Ellipsoids.} 
 \url{arXiv:2105.05295v1 [nlin.CD]} (2021). 

\bibitem{Tabach}
{S.\ Tabachnikov:} \textit{Geometry and Billiards.}
  American Mathematical Society, Providence/Rhode Island 2005. 
\end{thebibliography}
\end{document}